\documentclass[11pt,reqno]{amsart}
\usepackage[dvipdfm,a4paper,top=30mm,bottom=25mm,left=25mm,right=25mm]{geometry}
\usepackage{amsmath,amssymb,amsthm,mathtools,ulem,comment,mathrsfs,color}

\definecolor{refkey}{rgb}{1,0,0} 
\definecolor{labelkey}{rgb}{1,0,0}

\theoremstyle{definition}
\newtheorem{theorem}{Theorem}[section]
\newtheorem{proposition}[theorem]{Proposition}
\newtheorem{corollary}[theorem]{Corollary}
\newtheorem{lemma}[theorem]{Lemma}
\newtheorem{definition}[theorem]{Definition}

\newtheorem{remark}[theorem]{Remark}

\numberwithin{equation}{section} %(節, 番号)

\allowdisplaybreaks[4]

\def\rnum#1{\expandafter{\romannumeral #1}} 
\def\Rnum#1{\uppercase\expandafter{\romannumeral #1}} 

\makeatletter
    
    \@addtoreset{equation}{section}
  \makeatother

\title[Schr\"odinger equation with a potential]{Global dynamics below the ground state for the focusing Schr\"odinger equation with a potential}

\author[M. Hamano and M. Ikeda]{Masaru Hamano and Masahiro Ikeda}
\address[Masaru Hamano]{Department of Mathematics, Graduate School of Science and Engineering Saitama University, Shimo-Okubo 255, Sakura-ku, Saitama-shi, Saitama 338-8570, Japan}
\email{m.hamano.733@ms.saitama-u.ac.jp}

\address[Masahiro Ikeda]{Department of Mathematics, Faculty of Science and Technoligy, Keio University, 3-14-1 Hiyoshi, Kohoku-ku, Yokohama, 223-8522, Japan/Center for Advanced Intelligence Project, Riken, Japan}
\email{masahiro.ikeda@keio.jp/masahiro.ikeda@riken.jp}

\begin{document}

\maketitle

\textbf{Abstract.} In this paper, we consider the nonlinear Schr\"odinger equation with a real valued potential $V=V(x)$. We study global behavior of solutions to the equation with a data below the ground state under some conditions for the potential $V$ and prove a scattering result and a blowing-up result in mass-supercritical and energy-subcritical. Our proof of the scattering result is based on an argument by Dodson--Murphy \cite{03}. The proof of the blowing-up or growing-up result without radially symmetric assumption is based on the argument by Du--Wu--Zhang in \cite{04}. We can exclude the possibility of the growing-up result by the argument in \cite{16}, \cite{25}, and \cite{26} if ``the data and the potential are radially symmetric'' or ``the data has finite variance''.

\tableofcontents

\section{Introduction}

\subsection{Background}
　\\
　\\
In this paper, we consider time behavior of solutions to the following focusing mass-supercritical, energy-subcritical nonlinear Schr\"odinger equation with a linear potential.
\begin{equation}
\notag \text{(NLS$_V$) }
\begin{cases}
\hspace{-0.4cm}&\displaystyle{i\partial_tu+\Delta u-Vu+|u|^{p-1}u=0\quad(t,x)\in\mathbb{R}\times\mathbb{R}^3},\\
\hspace{-0.4cm}&\displaystyle{u(0,x)=u_0(x)\in H^1(\mathbb{R}^3)},
\end{cases}
\end{equation}
where $p>1$, $u=u(t,x)$ is a complex valued unknown function, $V=V(x)$ is a real valued given function of $x\in\mathbb{R}^3$.\\

(NLS$_V$) has physical background as follows. (NLS$_V$) with $V$ satisfying
\begin{align*}
V(x)\longrightarrow0\ \text{ as }\ |x|\rightarrow\infty\ \ \text{ and }\ \ V\in L^\infty
\end{align*}
is a model proposed to describe the local dynamics at a nucleation site, the attractive potential $V$ simulating a local depression in the ion density (see \cite{29}). Also, when $V$ is a harmonic potential $|x|^2$, (NLS$_V$) is a model to describe the Bose--Einstein condensate with attractive inter-particle interactions under a magnetic trap (see \cite{27}, \cite{28}, \cite{30}).\\

We define the potential class $\mathcal{K}_0$ as the norm closure of bounded and compactly supported functions with respect to the global Kato norm
\begin{align*}
\|V\|_{\mathcal{K}}:=\sup_{x\in\mathbb{R}^3}\int_{\mathbb{R}^3}\frac{|V(y)|}{|x-y|}dy,
\end{align*}
and denote the negative part of $V$ by
\begin{align*}
V_-(x):=\min(V(x),0).
\end{align*}
If we assume that
\begin{align}
V\in\mathcal{K}_0\cap L^\frac{3}{2} \label{002}
\end{align}
and
\begin{align}
\|V_-\|_{\mathcal{K}}<4\pi, \label{003}
\end{align}
then the Schr\"odinger operator $\mathcal{H}=-\Delta+V$ has no eigenvalues (see \cite{10}). The Schr\"odinger operator $\mathcal{H}$ is self-adjoint into $L^2$. By Stone's theorem, the Schr\"odinger evolution group $\{e^{-it\mathcal{H}}\}_{t\in\mathbb{R}}$ is generated on $L^2(\mathbb{R})$. Also, if $V$ satisfies \eqref{003}, then $\mathcal{H}$ and $1+\mathcal{H}$ are non-negative, that is, the following estimates
\begin{align*}
(\mathcal{H}f,f)_{L^2}=\|\nabla f\|_{L^2}^2+\int_{\mathbb{R}^3}V(x)|f(x)|^2dx\geq\left(1-\frac{\|V_-\|_{\mathcal{K}}}{4\pi}\right)\|\nabla f\|_{L^2}^2>0,
\end{align*}
\begin{align*}
((1+\mathcal{H})f,f)_{L^2}=\|f\|_{H^1}^2+\int_{\mathbb{R}^3}V(x)|f(x)|^2dx\geq\|f\|_{L^2}^2+\left(1-\frac{\|V_-\|_{\mathcal{K}}}{4\pi}\right)\|\nabla f\|_{L^2}^2>0
\end{align*}
hold for any $(0\neq)f\in H^1$ (see \cite{10}). Therefore, the fractional operators $(1+\mathcal{H})^\frac{1}{2}$ and $\mathcal{H}^\frac{1}{2}$ are well defined on the domain
\begin{align*}
H_V^1:=\{f\in H^1;\|f\|_{H_V^1}<\infty\}
\end{align*}
with the norm
\begin{align*}
\|f\|_{H_V^1}^2=\|(1+\mathcal{H})^\frac{1}{2}f\|_{L^2}^2=\|f\|_{L^2}^2+\|\mathcal{H}^\frac{1}{2}f\|_{L^2}^2:=\|f\|_{L^2}^2+\|\nabla f\|_{L^2}^2+\int_{\mathbb{R}^3}V(x)|f(x)|^2dx.
\end{align*}

We mainly assume $\frac{7}{3}<p<5$ in this paper. Here, we explain why the exponent $p$ is resricted to $\frac{7}{3}<p<5$. The equation (NLS$_V$) with $V=0$
\begin{align} 
i\partial_tu+\Delta u+|u|^{p-1}u=0\label{053}
\end{align}
is invariant under the following scaling transformation. If $u(t,x)$ is the solution to \eqref{053}, then for any $\lambda>0$,
\begin{align}
u_{[\lambda]}(t,x):=\lambda^\frac{2}{p-1}u(\lambda^2t,\lambda x) \label{004}
\end{align}
is also the solution to \eqref{053}. From transformation $u(t,x)\mapsto u_{[\lambda]}(t,x)$, the initial data is translated into 
\begin{align*}
u_0(x)\mapsto u_{0\{\lambda\}}(x):=\lambda^\frac{2}{p-1}u_0(\lambda x).
\end{align*}
$u_{0\{\lambda\}}$ and $s_c:=\frac{3}{2}-\frac{2}{p-1}$ satisfy
\begin{align}
\|u_{0\{\lambda\}}\|_{\dot{H}^{s_c}}=\|u_0\|_{\dot{H}^{s_c}}.\label{054}
\end{align}
From \eqref{054}, we see that if $p=\frac{7}{3}$ (resp. $p=5$), then $L^2$-norm (resp. $\dot{H}^1$-norm) of the initial data is invariant. In this sense, $p=\frac{7}{3}$ is called $L^2$ or mass-critical and $p=5$ is called $\dot{H}^1$ or energy-critical. $\frac{7}{3}<p<5$ is called mass-supercritical and energy-subcritical. Namely, we mainly consider mass-supercritical and energy-subcritical case in this paper.\\

The Cauchy problem (NLS$_V$) is locally well-posed in $H^1$ (see Theorem \ref{Local well-posedness 1}). $H^1$-solution to (NLS$_V$) conserves its mass and energy, defined respectively by
\begin{align*}
M[u(t)]:=\int_{\mathbb{R}^3}|u(t,x)|^2dx,
\end{align*}
\begin{align*}
E[u(t)]=E_V[u(t)]:=\frac{1}{2}\int_{\mathbb{R}^3}|\nabla u(t,x)|^2+V(x)|u(t,x)|^2dx-\frac{1}{p+1}\int_{\mathbb{R}^3}|u(t,x)|^{p+1}dx.
\end{align*}

\subsection{Main result}
　\\
　\\
To state our main result, we recall the definition of scattering, blowing-up, and growing-up to (NLS$_V$).

\begin{definition}[Scattering, Blowing-up, and Growing-up]
Let $(T_\text{min},T_\text{max})$ denote the maximal lifespan of $u$.
\begin{itemize}
\item (Scattering)\\
We say that the solution $u$ to (NLS$_V$) scatters in positive time (resp. negative time) if $T_\text{max}=\infty$ (resp. $T_\text{min}=-\infty$) and there exists $\psi_+\in H^1$ (resp. $\psi_-\in H^1$) such that
\begin{align*}
\lim_{t\rightarrow+\infty}\|u(t)-e^{-it\mathcal{H}}\psi_+\|_{H^1}=0\ \ \ \left(\text{resp.}\ \lim_{t\rightarrow-\infty}\|u(t)-e^{-it\mathcal{H}}\psi_-\|_{H^1}=0\right).
\end{align*}
\item (Blowing-up)\\
We say that the solution $u$ to (NLS$_V$) blows-up in positive time (resp. negative time) if $T_{\text{max}}<\infty$ (resp. $T_\text{min}>-\infty$).
\item (Glowing-up)\\
We say that the solution $u$ to (NLS$_V$) grows-up in positive time (resp. negative time) if $T_\text{max}=\infty$ (resp. $T_\text{min}=-\infty$) and
\begin{align*}
\limsup_{t\rightarrow+\infty}\|\nabla u(t)\|_{L^2}=\infty\ \ \ \left(\text{resp.}\ \limsup_{t\rightarrow-\infty}\|\nabla u(t)\|_{L^2}=\infty\right).
\end{align*}
\end{itemize}
\end{definition}

Our aim in this paper is to determine long time behavior of solutions to (NLS$_V$). There are various kinds of solutions depending on the choice of the data and the potential, for example, scattering solution, blowing-up solution, growing-up solution, standing wave solution and so on. In this paper, we investigate this problem under the assumption that a value of the product of mass and energy of the initial data is less than the product of mass and energy without the potential of the ground state (see Theorem \ref{Scattering versus blowing-up or growing-up}). Here, the ground state $Q$ is the unique radial positive solution to the following elliptic equation:
\begin{align}
-Q+\Delta Q+Q^p=0. \label{006}
\end{align}
The existence and uniqueness of radial positive solution to \eqref{006} were proved in \cite{02} and \cite{15}, respectively. We note that $u(t,x)=e^{it}Q(x)$ is a time-global non-scattering solution to \eqref{053} and is called a standing wave solution.\\

The following theorem is one of the main results in {\cite{10}.

\begin{theorem}[Hong, \cite{10}]\label{Hong theorem}
Let $p=3$ and $u_0\in H^1(\mathbb{R}^3)$. Suppose that $V$ satisfies $V\in \mathcal{K}_0\cap L^\frac{3}{2}$, $V\geq0$, $x\cdot\nabla V\leq0$, and $x\cdot\nabla V\in L^\frac{3}{2}$. We also assume that
\begin{align*}
M[u_0]E_V[u_0]<M[Q]E_0[Q]\ \ \text{ and }\ \ \|u_0\|_{L^2}\|\mathcal{H}^\frac{1}{2}u_0\|_{L^2}<\|Q\|_{L^2}\|\nabla Q\|_{L^2},
\end{align*}
where $E_0[Q]$ is the energy without a potential
\begin{align*}
E_0[Q]=\frac{1}{2}\int_{\mathbb{R}^3}|\nabla Q(x)|^2dx-\frac{1}{p+1}\int_{\mathbb{R}^3}|Q(x)|^{p+1}dx.
\end{align*}
Then,
\begin{align*}
\|u(t)\|_{L^2}\|\mathcal{H}^\frac{1}{2}u(t)\|_{L^2}<\|Q\|_{L^2}\|\nabla Q\|_{L^2}
\end{align*}
for any $t\in\mathbb{R}$ and $u$ scatters.
\end{theorem}

Natural questions arise from this theorem. It is whether a range of the exponent $p$ for nonlinearity can be extend or not. Moreover, It is whether we can determine behaviors of a solution to (NLS$_V$) with a data $u_0$ satisfying $\|u_0\|_{L^2}\|\mathcal{H}^\frac{1}{2}u_0\|_{L^2}>\|Q\|_{L^2}\|\nabla Q\|_{L^2}$ or not. We state our main result.

\begin{theorem}[Scattering versus blowing-up or growing-up]\label{Scattering versus blowing-up or growing-up}
Let $\frac{7}{3}<p<5$ and $u_0\in H^1(\mathbb{R}^3)$. Suppose that $V$ satisfies $V\geq0$, and $x\cdot\nabla V\in L^\frac{3}{2}$. We also assume that
\begin{align}
M[u_0]^{1-s_c}E_V[u_0]^{s_c}<M[Q]^{1-s_c}E_0[Q]^{s_c}, \label{008}
\end{align}
where $s_c=\frac{3}{2}-\frac{2}{p-1}$.
\begin{itemize}
\item[(1)] (Scattering)\\
If $V\in\mathcal{K}_0\cap L^\frac{3}{2}$, $x\cdot\nabla V\leq0$, and
\begin{align*}
\|u_0\|_{L^2}^{1-s_c}\|\nabla u_0\|_{L^2}^{s_c}<\|Q\|_{L^2}^{1-s_c}\|\nabla Q\|_{L^2}^{s_c},
\end{align*}
then $(T_\text{min},T_\text{max})=\mathbb{R}$, that is, $u$ exists globally in time and
\begin{align}
\|u(t)\|_{L^2}^{1-s_c}\|\nabla u(t)\|_{L^2}^{s_c}<\|Q\|_{L^2}^{1-s_c}\|\nabla Q\|_{L^2}^{s_c}\label{055}
\end{align}
for any $t\in\mathbb{R}$. Moreover, if $u_0$ and $V$ are radial, then $u$ scatters.
\item[(2)] (Blowing-up or growing-up)\\
If ``$V\in\mathcal{K}_0\cap L^\frac{3}{2}$ or $V\in L^\sigma$ for some $\sigma>\frac{3}{2}$'', $2V+x\cdot\nabla V\geq0$, and
\begin{align*}
\|u_0\|_{L^2}^{1-s_c}\|\mathcal{H}^\frac{1}{2}u_0\|_{L^2}^{s_c}>\|Q\|_{L^2}^{1-s_c}\|\nabla Q\|_{L^2}^{s_c},
\end{align*}
then
\begin{align}
\|u(t)\|_{L^2}^{1-s_c}\|\mathcal{H}^\frac{1}{2}u(t)\|_{L^2}^{s_c}>\|Q\|_{L^2}^{1-s_c}\|\nabla Q\|_{L^2}^{s_c}\label{056}
\end{align}
for any $t\in(T_\text{min},T_\text{max})$ and $u$ blows-up or grows-up. Furthermore, if $x\cdot\nabla V\geq0$ and the following (i) or (ii) holds:
\begin{itemize}
\item[(i)] ``$u_0$ and $V$ are radially symmetric'' and $V\in L^\sigma$ for some $\sigma>\frac{3}{2}$,
\item[(ii)] $xu_0\in L^2$ and ``$V\in\mathcal{K}_0\cap L^\frac{3}{2}$ or $V\in L^\sigma$ for some $\sigma>\frac{3}{2}$'',
\end{itemize}
then $u$ blows-up.
\end{itemize}
\end{theorem}

\begin{corollary}\label{Mass-critical result}
We prove the similar blowing-up result in the mass-critical case $p=\frac{7}{3}$. We assume that the potential $V$ satisfies the same assumptions as in Theorem \ref{Scattering versus blowing-up or growing-up} (2). The initial data $u_0\in H^1$ satisfies $E_V[u_0]<0$ (instead of \eqref{008}). Then, the same conclusion as Theorem \ref{Scattering versus blowing-up or growing-up} (2) holds.
\end{corollary}

\begin{remark}
We comment the assumptions of $V$. If $V$ satisfies the condition of the blowing-up part, that is, $V$ is radial, $V\geq0$, and $x\cdot\nabla V\geq0$ implies $V\notin L^\frac{3}{2}$. Indeed, since $x\cdot\nabla V=rV'$, we have
\begin{align*}
2V+rV'\geq0\ 
	&\Leftrightarrow\frac{1}{V}\cdot\frac{dV}{dr}\geq-\frac{2}{r}\ \Rightarrow
\ \int_1^r\frac{1}{V}dV\geq-\int_1^r\frac{2}{r}dr\\
	&\Leftrightarrow\ \log V(r)-\log V(1)\geq-2\log r\ \Leftrightarrow\ 
\log V(r)\geq\log\frac{V(1)}{r^2}\ \Leftrightarrow\ V(r)\geq\frac{V(1)}{r^2}
\end{align*}
for any $r\geq1$. From this fact, we do not get the blowing-up result under $V\in\mathcal{K}_0\cap L^\frac{3}{2}$.
\end{remark}

\begin{remark}
We compare our result (Theorem \ref{Scattering versus blowing-up or growing-up}) with Hong's result (Theorem \ref{Hong theorem}). Theorem \ref{Scattering versus blowing-up or growing-up} extends a range of the exponent $p$ for nonlinearity. In Theorem \ref{Scattering versus blowing-up or growing-up}, it is assumed that $u_0$ and $V$ are radial in scattering part. We characterize sufficient condition of scattering by $\|\nabla u_0\|_{L^2}$ not $\|\mathcal{H}^\frac{1}{2}u_0\|_{L^2}$. Since $\|\nabla u_0\|_{L^2}\leq\|\mathcal{H}^\frac{1}{2}u_0\|_{L^2}$ holds by $V\geq0$, our result extends Theorem \ref{Hong theorem} in this point. Theorem \ref{Scattering versus blowing-up or growing-up} also contains a blowing-up or growing-up result and a blowing-up result.
\end{remark}

\subsection{Strategy and idea of proof}
　\\
　\\
Hong \cite{10} studied scattering to (NLS$_V$) without radially symmetry assumption via Kenig--Merle's type approach (linear profile decomposition, construction of critical solution, rigidity, and so). To simplify Hong's argument, we use the argument by Dodson--Murphy in \cite{03}. In order to use Dodson--Murphy's argument, we assume that $u_0$ and $V$ are radially symmetric. We characterize the sufficient condition of scattering with $\|\nabla u_0\|_{L^2}$ not $\|\mathcal{H}^\frac{1}{2}u_0\|_{L^2}$ by using $V\geq0$. More precisely, in this improvement, it is important that we deduce Proposition \ref{Coercivity1} (i) with $\|\nabla u_0\|_{L^2}$.\\
The proof of the blowing-up or growing-up result without a radially symmetric assumption is based on the argument by Du--Wu--Zhang \cite{04} with a time-independent estimate of a functional (see Lemma \ref{Lemma3 for blows-up or grows-up}).\\
The proof of a blowing-up result with the radially symmetric assumption is based on \cite[{\S}4.1]{16} and \cite[{\S}2]{25}. The argument is originally established by Ogawa--Tsutsumi \cite{17}. The proof of a blowing-up result with the finite variance assumption is based on \cite{26}.

\subsection{Known results}
　\\

In the past ten years, global behavior of solutions below the ground state for (NLS$_V$) was studied by several authors. First, we introduce results for (NLS$_V$) with $V=0$. Kenig--Merle in \cite{12} showed a scattering result and blowing-up result under $p=1+\frac{4}{N-2}$, $N=3,4,5$, and $u_0\in\dot{H}_\text{rad}^1(\mathbb{R}^N)$. Holmer--Roudenko in \cite{09} showed a scattering result and a blowing-up result under $p=3$, $u_0\in H_\text{rad}^1(\mathbb{R}^3)$. Duyckaerts--Holmer--Roudenko in \cite{05} removed the radial condition for the scattering result in \cite{09}. Fang--Xie--Cazenave in \cite{06} showed a scattering result, and Akahori--Nawa in \cite{01} showed a scattering result and a blowing-up or growing-up result in the mass-supercritical and energy-subcritical $1+\frac{4}{N}<p<1+\frac{4}{N-2}$, $u_0\in H^1(\mathbb{R}^N)$. Du--Wu--Zhang in \cite{04} showed a blowing-up or growing-up result under mass-supercritical and energy-subcritical. Dodson--Murphy in \cite{03} gave a new proof for the scattering result under the same setting as \cite{09}.\\
Next, we introduce results for (NLS$_V$) with $V\neq0$. Hong in \cite{10} showed a scattering result under $p=3$, $u_0\in H^1(\mathbb{R}^3)$, $V\in\mathcal{K}_0(\mathbb{R}^3)\cap L^\frac{3}{2}(\mathbb{R}^3)$, $V\geq0$, and $x\cdot\nabla V\leq0$. Killip--Murphy--Visan--Zheng in \cite{13} showed a scattering result and a blowing-up result under $p=3$, $u_0\in H^1(\mathbb{R}^3)$, $V=\frac{a}{|x|^2}$, and $a>-\frac{1}{4}$. Lu--Miao--Murphy in \cite{14} showed a scattering result and a blowing-up result under $1+\frac{4}{N}<p<1+\frac{4}{N-2}$, $3\leq N\leq6$, $u_0\in H^1(\mathbb{R}^N)$, $V=\frac{a}{|x|^2}$, and
\begin{equation}
\notag a>
\begin{cases}
&\hspace{-0.4cm}{-\frac{1}{4},\quad\left(N=3,\ \frac{7}{3}<p\leq3\right)},\\
&\hspace{-0.4cm}{-\frac{1}{4}+\left(\frac{1}{2}-\frac{1}{p-1}\right)^2,\quad(N=3,\ 3<p<5),}\\
&\hspace{-0.4cm}{-\left(\frac{N-2}{2}\right)^2+\left(\frac{N-2}{2}-\frac{1}{p-1}\right)^2,\quad(4\leq N\leq 6).}
\end{cases}
\end{equation}
Zheng in \cite{20} showed a scattering result under $1+\frac{4}{N}<p<1+\frac{4}{N-2}$, $N\geq3$, $u_0\in H_\text{rad}^1(\mathbb{R}^N)$, $V=\frac{a}{|x|^2}$, and $a>-\frac{1}{4}$ $(N=3\text{ and }\frac{7}{3}<p\leq3)$, $a>-\frac{1}{4}+(\frac{N-2}{2}-\frac{1}{p-1})^2$ $(N=3\text{ and }3<p<5)$, $a>-(\frac{N-2}{2})^2+(\frac{N-2}{2}-\frac{1}{p-1})^2$ $(N\geq4)$. Ikeda in \cite{11} showed a scattering result under $p>5$, $V$, $V'\in L_1^1(\mathbb{R}):=\{f\in L^1(\mathbb{R});(1+|x|)f\in L^1(\mathbb{R})\}$, $xV'\in L^1(\mathbb{R})+L^\infty(\mathbb{R})$, and $xV'\leq0$ and a blowing-up result under $p>5$, $V\in L^1(\mathbb{R})+L^\infty(\mathbb{R})$, and $xV'+2V\geq0$.

\subsection{Organization of the paper}
　\\

The organization of this paper is as follows. In section 2, we collect some definitions and some elementary tools. Also, we establish local well-posedness in $H^1$ of (NLS$_V$). In section 3, we prove the scattering result in Theorem \ref{Scattering versus blowing-up or growing-up}. In section 4, we prove the blowing-up or growing-up result in Theorem \ref{Scattering versus blowing-up or growing-up}. In section 5, we prove blowing-up result in Theorem \ref{Scattering versus blowing-up or growing-up}.

\section{Preliminaries}

In this section, we define some notations and collect some known tools. 

\subsection{Notation and definition}
　\\

For nonnegative $X$ and $Y$, we write $X\lesssim Y$ to denote $X\leq CY$ for some $C>0$. If $X\lesssim Y\lesssim X$ holds, we write $X\sim Y$. The dependence of implicit constants on parameters will be indicated by subscripts, e.g. $X\lesssim_uY$ denotes $X\leq CY$ for some $C=C(u)$. We write $a'\in[1,\infty]$ to denote the H\"older dual exponent to $a\in[1,\infty]$, that is, the solution $\frac{1}{a}+\frac{1}{a'}=1$.\\

For $1\leq p\leq\infty$, $L^p=L^p(\mathbb{R}^3)$ denotes the usual Lebesgue space. For a Banach space $X$, we use $L^q(I;X)$ to denote the Banach space of functions $f:I\times\mathbb{R}^3\longrightarrow\mathbb{C}$ whose norm is
\begin{align*}
\|f\|_{L^q(I;X)}:=\left(\int_I\|f(t)\|_X^qdt\right)^\frac{1}{q}<\infty,
\end{align*}
with the usual modificafion when $q=\infty$. We extend our notation as follows: If a time interval is not spacified, then the $t$-norm is evaluated over $(-\infty,\infty)$. To indicate a restriction to a time subinterval $I\subset(-\infty,\infty)$, we will write as $L^q(I)$.\\

We define the Fourier transform of $f$ on $\mathbb{R}^3$ by
\begin{align*}
\mathcal{F}f(\xi)=\widehat{f}(\xi):=\int_{\mathbb{R}^3}e^{-2\pi ix\cdot\xi}f(x)dx
\end{align*}
and define the inverse Fourier transform of $f$ on $\mathbb{R}^3$ by
\begin{align*}
\mathcal{F}^{-1}f(x)=\check{f}(x):=\int_{\mathbb{R}^3}e^{2\pi ix\cdot\xi}f(\xi)d\xi,
\end{align*}
where $x\cdot\xi$ denotes the usual inner product of $x$ and $\xi$ on $\mathbb{R}^3$.\\

$W^{s,p}=(1-\Delta)^{-\frac{s}{2}}L^p$ and $\dot{W}^{s,p}=(-\Delta)^{-\frac{s}{2}}L^p$ are the inhomogeneous Sobolev space and the homogeneous Sobolev space, respectively for $s\in\mathbb{R}$ and $p\in[1,\infty]$, where $(1-\Delta)^\frac{s}{2}=\mathcal{F}^{-1}(1+4\pi^2|\xi|^2)^\frac{s}{2}\mathcal{F}$ and $(-\Delta)^\frac{s}{2}=\mathcal{F}^{-1}(2\pi|\xi|)^s\mathcal{F}$, respectively. When $p=2$, we express $W^{s,2}=H^s$ and $\dot{W}^{s,2}=\dot{H}^s$. We also define Sobolev space with a potential by $W_V^{s,p}=(1+\mathcal{H})^{-\frac{s}{2}}L^p$ and $\dot{W}_V^{s,p}=\mathcal{H}^{-\frac{s}{2}}L^p$ for $s\in\mathbb{R}$ and $p\in[1,\infty]$. When $p=2$, we express $W_V^{s,2}=H_V^s$ and $\dot{W}_V^{s,2}=\dot{H}_V^s$.\\

We introduce a cutoff function which is used throughout this paper. We define
\begin{equation}
\notag \chi(x)=
\begin{cases}
&\hspace{0.1cm}1\hspace{1.1cm}(0\leq|x|\leq\frac{1}{2}),\\
&\hspace{-0.4cm}\text{smooth}\hspace{0.55cm}(\frac{1}{2}\leq|x|\leq1),\\
&\hspace{0.1cm}0\hspace{1.1cm}(1\leq|x|)
\end{cases}
\end{equation}
and define
\begin{align}
\chi_R(x)=\chi\left(\frac{x}{R}\right)\label{072}
\end{align}
for $R>0$.

\subsection{Some tools}

\begin{lemma}[Norm equivalence, \cite{10}]\label{Norm equivalence}
If $V\in\mathcal{K}_0\cap L^\frac{3}{2}$ and $\|V_-\|_{\mathcal{K}}<4\pi$, then
\begin{align*}
\|\mathcal{H}^\frac{s}{2}f\|_{L^r}\sim\|f\|_{\dot{W}^{s,r}},\ \ \ \|(1+\mathcal{H})^\frac{s}{2}f\|_{L^r}\sim\|f\|_{W^{s,r}},
\end{align*}
where $1<r<\frac{3}{s}$ and $0\leq s\leq2$.
\end{lemma}

\begin{lemma}[Sobolev inequality, \cite{10}]\label{Sobolev inequality}
If $V\in\mathcal{K}_0\cap L^\frac{3}{2}$ and $\|V_-\|_{\mathcal{K}}<4\pi$, then
\begin{align*}
\|f\|_{L^q}\lesssim\|\mathcal{H}^\frac{s}{2}f\|_{L^p},\ \ \ \|f\|_{L^q}\lesssim\|(1+\mathcal{H})^\frac{s}{2}f\|_{L^p},
\end{align*}
where $1<p<q<\infty$, $1<p<\frac{3}{s}$, $0\leq s\leq2$ and $\frac{1}{q}=\frac{1}{p}-\frac{s}{3}$.
\end{lemma}

\begin{proposition}[Gagliardo-Nirenberg inequliaty without a potential, \cite{23}]
Let $1\leq p\leq5$. For $f\in H^1(\mathbb{R}^3)$, the estimate
\begin{align}
\|f\|_{L^{p+1}}^{p+1}\leq C_\text{GN}\|f\|_{L^2}^{\frac{5-p}{2}}\|\nabla f\|_{L^2}^{\frac{3(p-1)}{2}} \label{005}
\end{align}
holds, where $C_\text{GN}$ is the sharp constant and $C_\text{GN}$ is attained by the ground state $Q$ (defined by \eqref{006}), that is,
\begin{align*}
C_\text{GN}=\frac{\|Q\|_{L^{p+1}}^{p+1}}{\|Q\|_{L^2}^{\frac{5-p}{2}}\|\nabla Q\|_{L^2}^{\frac{3(p-1)}{2}}}.
\end{align*}
\end{proposition}

\begin{theorem}[Dispersive estimate, \cite{10}]\label{Dispersive estimate}
If $V\in\mathcal{K}_0\cap L^\frac{3}{2}$ and $\|V_-\|_\mathcal{K}<4\pi$, then
\begin{align*}
\|e^{-it\mathcal{H}}f\|_{L^\infty}\lesssim\frac{1}{|t|^{\frac{3}{2}}}\|f\|_{L^1}.
\end{align*}
\end{theorem}

\begin{definition}[$\dot{H}^s$\,admissible and Strichartz norm]
We say that a pair of exponents $(q,r)$ is called $\dot{H}^s$\,admissible in three dimensions if $2\leq q,r\leq\infty$ and
\begin{align*}
\frac{2}{q}+\frac{3}{r}=\frac{3}{2}-s.
\end{align*}
We define Strichartz norm by
\begin{align*}
\|u\|_{S(L^2)}:=\sup_{\substack{(q,r):L^2 \text{admissible}\\ 2\leq q\leq\infty, 2\leq r\leq 6}}\|u\|_{L^qL^r}
\end{align*}
and its dual norm by
\begin{align*}
\|u\|_{S'(L^2)}:=\inf_{\substack{(q,r):L^2 \text{admissible}\\2\leq q\leq \infty, 2\leq r\leq 6}}\|u\|_{L^{q'}L^{r'}}.
\end{align*}
\end{definition}

\begin{theorem}[Strichartz estimate, \cite{07}, \cite{21}]\label{Strichartz estimate}
If $V\in\mathcal{K}_0\cap L^\frac{3}{2}$ and $\|V_-\|_{\mathcal{K}}<4\pi$, then the following estimates hold.
\begin{itemize}
\item (Homogeneous estimates)
\begin{align}
\|e^{-it\mathcal{H}}f\|_{S(L^2)}\lesssim\|f\|_{L^2}. \label{010}
\end{align}
If $(q,r)$ is $\dot{H}^{s_c}$\,admissible and is in a set $\Lambda_{s_c}$ defined as
\begin{equation*}
\Lambda_{s_c}:=
\begin{cases}
&\hspace{-0.4cm}\displaystyle{\left\{(q,r)\,;\ 2\leq q\leq \infty,\ \frac{6}{3-2s_c}\leq r\leq\frac{6}{1-2s_c}\ \,\right\}\ \left(0<s_c<\frac{1}{2}\right),}\\[0.4cm]
&\hspace{-0.4cm}\displaystyle{\left\{(q,r)\,;\ \frac{4}{3-2s_c}<q\leq\infty,\ \frac{6}{3-2s_c}\leq r<\infty\right\}\ \left(\frac{1}{2}\leq s_c<1\right),}
\end{cases}
\end{equation*}
then
\begin{align}
\|e^{-it\mathcal{H}}f\|_{L^qL^r}\lesssim\|f\|_{\dot{H}^{s_c}}. \label{012}
\end{align}
\item (Inhomogeneous estimates)
\begin{align}
\left\|\int_0^te^{-i(t-s)\mathcal{H}}F(\cdot,s)ds\right\|_{S(L^2)}\lesssim\|F\|_{S'(L^2)}. \label{011}
\end{align}
If $(q,r)$ is $\dot{H}^{s_c}$\,admissible and is in a set $\Lambda_{s_c}$, then
\begin{align}
\left\|\int_0^te^{-i(t-s)\mathcal{H}}F(\cdot,s)ds\right\|_{L^qL^r}\lesssim\||\nabla|^{s_c}F\|_{S'(L^2)}, \label{013}
\end{align}
where implicit constants are independent of $f$ and $F$. Even if time is restricted, Theorem \ref{Strichartz estimate} still holds.
\end{itemize}
\end{theorem}

\eqref{010} and \eqref{011} are cited in \cite{07}. \eqref{012} and \eqref{013} are deduced by combining \eqref{010}, \eqref{011}, and Lemma \ref{Sobolev inequality}. 

\begin{proposition}[Pohozaev identities without a potential]\label{Pohozaev identity}
Let $1<p<5$. The ground state $Q$ for the elliptic equation \eqref{006} satisfies the following Pohozaev identities.
\begin{align*}
\|Q\|_{L^{p+1}}^{p+1}=\frac{2(p+1)}{5-p}\|Q\|_{L^2}^2,\ \ \ \|Q\|_{L^{p+1}}^{p+1}=\frac{2(p+1)}{3(p-1)}\|\nabla Q\|_{L^2}^2.
\end{align*}
\end{proposition}

The proof of this proposition, see \cite[Lemma 8.1.2]{02}.\\
　\\
Using Proposition \ref{Pohozaev identity}, we have
\begin{align}
E_0[Q]=\frac{3p-7}{6(p-1)}\|\nabla Q\|_{L^2}^2\ \ \text{ and }\ \ C_\text{GN}=\frac{2(p+1)}{3(p-1)}\frac{1}{\|Q\|_{L^2}^\frac{5-p}{2}\|\nabla Q\|_{L^2}^\frac{3p-7}{2}}. \label{021}
\end{align}

\begin{lemma}[Fractional calculus, \cite{19}]\label{Fractional calculus}
Suppose $G\in C^1(\mathbb{C})$ and $s\in(0,1]$. Let $1<r, r_2<\infty$ and $1<r_1\leq\infty$ satisfying $\frac{1}{r}=\frac{1}{r_1}+\frac{1}{r_2}$. Then, we have
\begin{align*}
\||\nabla|^sG(u)\|_{L^r}\lesssim\|G'(u)\|_{L^{r_1}}\||\nabla|^su\|_{L^{r_2}}.
\end{align*}
\end{lemma}

\begin{lemma}[Radial Sobolev inequality]\label{Radial Sobolev inequality}
Let $1\leq p$. For a radial function $f\in H^1$, it follows that
\begin{align*}
\|f\|_{L^{p+1}(R\leq|x|)}^{p+1}\lesssim\frac{1}{R^{p-1}}\|f\|_{L^2(R\leq|x|)}^\frac{p+3}{2}\|\nabla f\|_{L^2(R\leq|x|)}^\frac{p-1}{2}
\end{align*}
for any $R>0$, where the implicit constant is independent of $R$ and $f$.
\end{lemma}

\begin{proof}
By \cite[Lemma 1]{17}, we have
\begin{align*}
\|f\|_{L^{p+1}(R\leq|x|)}^{p+1}&\leq\|f\|_{L^\infty(R\leq|x|)}^{p-1}\|f\|_{L^2(R\leq|x|)}^2\lesssim\frac{1}{R^{p-1}}\|f\|_{L^2(R\leq|x|)}^\frac{p+3}{2}\|\nabla f\|_{L^2(R\leq|x|)}^\frac{p-1}{2}.
\end{align*}
\end{proof}

\begin{lemma}[Hardy's inequality, \cite{24}]
Let $q\in[0,2]$. If $f\in H^1$, then
\begin{align*}
\int_{\mathbb{R}^3}\frac{|f(x)|^2}{|x|^q}dx\leq\left(\frac{2}{3-q}\right)^q\|f\|_{L^2}^{2-q}\|\nabla f\|_{L^2}^q.
\end{align*}
\end{lemma}

\begin{lemma}[Radial Sobolev embedding, \cite{18}, \cite{20}]\label{Radial Sobolev embedding}
For a radial function $f\in H^1$ and $\frac{1}{2}\leq s\leq1$, it follows that
\begin{align*}
\||x|^sf\|_{L^\infty}\lesssim\|f\|_{H^1}.
\end{align*}
\end{lemma}

This estimate is cited in \cite{18} and \cite{20}. For convenience of readers, we give its proof.

\begin{proof}
We set $r=|x|$. Using Hardy's inequality, we have
\begin{align*}
r^{2s}|f(r)|^2
	&=-\int_r^\infty\frac{d}{d\tau}(\tau^{2s}|f(\tau)|^2)d\tau=-\int_r^\infty(2s\tau^{2s-1}|f(\tau)|^2+2\text{Re}(\tau^{2s}f(\tau)\overline{f'(\tau)}))d\tau\\
	&\leq2\int_r^\infty\tau^{2s}|f(\tau)||f'(\tau)|d\tau\leq2\left(\int_r^\infty \tau^{2(2s-1)}|f(\tau)|^2d\tau\right)^\frac{1}{2}\left(\int_r^\infty\tau^2|f'(\tau)|^2d\tau\right)^\frac{1}{2}\\
	&\lesssim\left(\int_{\mathbb{R}^3}\frac{|f(x)|^2}{|x|^{4(1-s)}}dx\right)^\frac{1}{2}\left(\int_{\mathbb{R}^3}|\nabla f(x)|^2dx\right)^\frac{1}{2}\lesssim\|f\|_{H^1}^2.
\end{align*}
\end{proof}

\begin{proposition}[Virial identity, \cite{10}]\label{Finite variance}
For a solution $u(t)$ to (NLS$_V$) satisfying $xu_0\in L^2$, we define
\begin{align*}
I(t):=\int_{\mathbb{R}^3}|x|^2|u(t,x)|^2dx.
\end{align*}
Using (NLS$_V$), one finds
\begin{align*}
I'(t)=4\text{Im}\int_{\mathbb{R}^3}\overline{u}\nabla u\cdot xdx,
\end{align*}
\begin{align*}
I''(t)=8\int_{\mathbb{R}^3}|\nabla u|^2dx-\frac{12(p-1)}{p+1}&\int_{\mathbb{R}^3}|u|^{p+1}dx-4\int_{\mathbb{R}^3}(x\cdot\nabla V)|u|^2dx.
\end{align*}
\end{proposition}

\begin{proposition}[Localized virial identity, \cite{10}]\label{Virial identity}
Given a real valued weight function $\omega\in C^\infty(\mathbb{R}^3)$ and the solution $u(t)$ to (NLS$_V$), we define
\begin{align*}
I(t):=\int_{\mathbb{R}^3}\omega(x)|u(t,x)|^2dx.
\end{align*}
Using (NLS$_V$), one finds
\begin{align}
I'(t)=2\text{Im}\int_{\mathbb{R}^3}\overline{u}\nabla u\cdot\nabla\omega dx, \label{022}
\end{align}
\begin{align}
I''(t)=4\sum_{i,j=1}^3\text{Re}\int_{\mathbb{R}^3}\omega_{ij}u_i\overline{u_j}dx-\frac{2(p-1)}{p+1}
	&\int_{\mathbb{R}^3}\Delta\omega|u|^{p+1}dx \notag \\
	&-\int_{\mathbb{R}^3}\Delta^2\omega|u|^2dx-2\int_{\mathbb{R}^3}(\nabla\omega\cdot\nabla V)|u|^2dx. \label{023}
\end{align}
If $\omega$ is radial, then we can write
\begin{align}
I'(t)=2\text{Im}\int_{\mathbb{R}^3}\frac{x\cdot\nabla u}{r}\overline{u}\omega'dx, \label{024}
\end{align}
\begin{align}
	&I''(t)=4\int_{\mathbb{R}^3}\left(\frac{\omega''}{r^2}-\frac{\omega'}{r^3}\right)|x\cdot\nabla u|^2dx+4\int_{\mathbb{R}^3}\frac{\omega'}{r}|\nabla u|^2dx-\frac{2(p-1)}{p+1}\int_{\mathbb{R}^3}\left(\omega''+\frac{2}{r}\omega'\right)|u|^{p+1}dx \notag \\
	&\hspace{5.0cm}-\int_{\mathbb{R}^3}\left(\omega^{(4)}+\frac{4}{r}\omega^{(3)}\right)|u|^2dx-2\int_{\mathbb{R}^3}\frac{\omega'}{r}(x\cdot\nabla V)|u|^2dx. \label{025}
\end{align}
\end{proposition}

Proposition \ref{Virial identity} with $p=3$ was proved in \cite{10}. Proposition \ref{Virial identity} follows by a direct calculation. 

\section{Local well-posedness}

In this section, we investigate local well-posedness of (NLS$_V$).

\begin{theorem}[Local well-posedness in $H^1$]\label{Local well-posedness 1}
Let $1\leq p<5$, $V\in\mathcal{K}_0\cap L^\frac{3}{2}$, and $\|V_-\|_{\mathcal{K}}<4\pi$. For any $u_0\in H^1$, there exists $T=T(\|u_0\|_{H^1},V)>0$ and a unique solution $u$ to (NLS$_V$) on a time interval $I=[-T,T]$ with
\begin{align*}
u\in C([-T,T];H^1)\cap L^q(-T,T;W_V^{1,r})
\end{align*}
for any $L^2$\,admissible pair $(q,r)$.
\end{theorem}

\begin{proof}
We define a function space
\begin{align*}
E:=\left\{u:\|(1+\mathcal{H})^\frac{1}{2}u\|_{S(L^2;I)}\leq2c\|u_0\|_{H^1}\right\}
\end{align*}
and a metric
\begin{align*}
d(u,v):=\|u-v\|_{S(L^2;I)}
\end{align*}
for $u,v\in E$. Then, $(E,d)$ is a complete metric space. We set a map
\begin{align*}
\Phi_{u_0}(u):=e^{-it\mathcal{H}}u_0+i\int_0^te^{-i(t-s)\mathcal{H}}(|u|^{p-1}u)(s)ds
\end{align*}
for $u\in E$. We take a number $\beta$ satisfying
\begin{equation*}
\notag \max\left\{2,\frac{4}{5-p}\right\}<\beta<
\begin{cases}
&\hspace{-0.4cm}\displaystyle{\frac{4}{3-p}}\hspace{0.5cm}\left(1<p<3\right),\\
&\hspace{-0.15cm}\infty\hspace{0.77cm}\left(3\leq p<5\right).
\end{cases}
\end{equation*}
Using Proposition \ref{Strichartz estimate} and Sobolev's embedding, we have
\begin{align*}
\|(1+\mathcal{H})^\frac{1}{2}\Phi_{u_0}(u)\|_{S(L^2;I)}
	&\leq\|e^{-it\mathcal{H}}(1+\mathcal{H})^\frac{1}{2}u_0\|_{S(L^2;I)}+\left\|\int_0^te^{-i(t-s)\mathcal{H}}(1+\mathcal{H})^\frac{1}{2}(|u|^{p-1}u)(s)ds\right\|_{S(L^2;I)}\\
	&\leq c\|(1+\mathcal{H})^\frac{1}{2}u_0\|_{L^2}+c\|(1+\mathcal{H})^\frac{1}{2}(|u|^{p-1}u)\|_{L^2(I;L^\frac{6}{5})}\\
	&\leq c\|u_0\|_{H^1}+c\||u|^{p-1}u\|_{L^2(I;W^{1,\frac{6}{5}})}\\
	&\leq c\|u_0\|_{H^1}+c\left\|\||u|^{p-1}\|_{L^3}\|u\|_{H^1}\right\|_{L^2(I)}\\
	&\leq c\|u_0\|_{H^1}+c\left\|\|u\|_{L^{3(p-1)}}^{p-1}\|(1+\mathcal{H})^\frac{1}{2}u\|_{L^2}\right\|_{L^2(I)}\\
%	&\leq c\|u_0\|_{H^1}+c\|1\|_{L^\beta(I)}\left\|\|u\|_{L^{3(p-1)}}^{p-1}\right\|_{L^\frac{2\beta}{\beta-2}(I)}\|(1+\mathcal{H})^\frac{1}{2}u\|_{L^\infty(I;L^2)}\\
	&\leq c\|u_0\|_{H^1}+c\,T^\frac{1}{\beta}\|u\|_{L^\frac{2\beta(p-1)}{\beta-2}(I;L^{3(p-1)})}^{p-1}\|(1+\mathcal{H})^\frac{1}{2}u\|_{L^\infty(I; L^2)}\\
	&\leq c\|u_0\|_{H^1}+c\,T^\frac{1}{\beta}\|u\|_{L^\frac{2\beta(p-1)}{\beta-2}(I;W^{1,\frac{6(p-1)\beta}{\beta(3p-5)+4}})}^{p-1}\|(1+\mathcal{H})^\frac{1}{2}u\|_{L^\infty(I;L^2)}\\
	&\leq c\|u_0\|_{H^1}+c\,T^\frac{1}{\beta}\|(1+\mathcal{H})^\frac{1}{2}u\|_{L^\frac{2\beta(p-1)}{\beta-2}(I;L^\frac{6(p-1)\beta}{\beta(3p-5)+4})}^{p-1}\|(1+\mathcal{H})^\frac{1}{2}u\|_{L^\infty(I;L^2)}\\
	&\leq c\|u_0\|_{H^1}+c\,T^\frac{1}{\beta}\|(1+\mathcal{H})^\frac{1}{2}u\|_{S(L^2;I)}^p\\
%	&\leq c\|u_0\|_{H^1}+c(2c)^pT^\frac{1}{\beta}\|u_0\|_{H^1}^p\\
	&\leq \left(1+(2c)^pT^\frac{1}{\beta}\|u_0\|_{H^1}^{p-1}\right)c\|u_0\|_{H^1}
\end{align*}
and
\begin{align*}
\|\Phi_{u_0}(u)-\Phi_{u_0}(v)\|_{S(L^2;I)}
	&\leq\left\|\int_0^te^{-i(t-s)\mathcal{H}}(|u|^{p-1}u-|v|^{p-1}v)(s)ds\right\|_{S(L^2;I)}\\
	&\leq c\||u|^{p-1}u-|v|^{p-1}v\|_{L^2(I;L^\frac{6}{5})}\\
	&\leq c\|(|u|^{p-1}+|v|^{p-1})|u-v|\|_{L^2(I;L^\frac{6}{5})}\\
	&\leq c\left\|\left\||u|^{p-1}+|v|^{p-1}\right\|_{L^3}\|u-v\|_{L^2}\right\|_{L^2(I)}\\
	&\leq c\left\|\left(\|u\|_{L^{3(p-1)}}^{p-1}+\|v\|_{L^{3(p-1)}}^{p-1}\right)\|u-v\|_{L^2}\right\|_{L^2(I)}\\
	&\leq c\,T^\frac{1}{\beta}\left(\|u\|_{L^\frac{2\beta(p-1)}{\beta-2}(I;L^{3(p-1)})}^{p-1}+\|v\|_{L^\frac{2\beta(p-1)}{\beta-2}(I;L^{3(p-1)})}^{p-1}\right)\|u-v\|_{L^\infty(I;L^2)}\\
	&\leq c\,T^\frac{1}{\beta}\left(\|(1+\mathcal{H})^\frac{1}{2}u\|_{S(L^2;I)}^{p-1}+\|(1+\mathcal{H})^\frac{1}{2}v\|_{S(L^2;I)}^{p-1}\right)\|u-v\|_{S(L^2;I)}\\
	&\leq (2c)^pT^\frac{1}{\beta}\|u_0\|_{H^1}^{p-1}\|u-v\|_{S(L^2;I)}.
\end{align*}
Here, we take $T>0$ sufficiently small such as
\begin{align*}
(2c)^pT^\frac{1}{\beta}\|u_0\|_{H^1}^{p-1}<1.
\end{align*}
Then, $\Phi_{u_0}$ is a contraction map on $E$ and there exists a unique solution to (NLS$_V$) in $[-T,T]$.
\end{proof}

The following theorem holds  by \cite[Theorem 4.3.1]{02}.

\begin{theorem}
Let $1\leq p<5$, $V\in L^\sigma$ for some $\sigma>\frac{3}{2}$. For any $u_0\in H^1$, there exists $T=T(u_0,V)>0$ and a unique solution to (NLS$_V$) on a time interval $I=[-T,T]$.
\end{theorem}

\begin{theorem}[Local well-posedness in $H^1\cap |x|^{-1}L^2$]\label{Local well-posedness 2}
Let $1\leq p<5$. Let ``$V\in\mathcal{K}_0\cap L^\frac{3}{2}$ and $\|V_-\|_{\mathcal{K}}<4\pi$'' or ``$V\in L^\sigma$ for some $\sigma>\frac{3}{2}$''. Let $I\ni0$ be a time interval, $u_0\in H^1$, and $u$ be a $H^1$-solution to (NLS$_V$) on $I$. If $|\cdot|u_0\in L^2$, then a map $t\mapsto|\cdot|u(t,\cdot)$ belongs to $C(I;L^2)$.
\end{theorem}

The proof of this theorem is based on the argument in \cite[Lemma 6.5.2]{02}.

\begin{proof}
We set $I=[0,T)$ with $0<T\leq\infty$. Let $\varepsilon>0$. We define a function
\begin{align*}
f_\varepsilon(t)=\|e^{-\varepsilon|x|^2}|x|u(t)\|_{L^2}^2.
\end{align*}
Then, we have
\begin{align*}
f'_{\varepsilon}(t)
	&=\frac{d}{dt}\int_{\mathbb{R}^3}e^{-2\varepsilon|x|^2}|x|^2|u(t,x)|^2dt=2\text{Re}\int_{\mathbb{R}^3}e^{-2\varepsilon|x|^2}|x|^2\partial_tu\overline{u}dx\\
	&=2\text{Re}\int_{\mathbb{R}^3}e^{-2\varepsilon|x|^2}|x|^2(i\Delta u-iVu+i|u|^{p-1}u)\overline{u}dx=2\text{Im}\int_{\mathbb{R}^3}\nabla(e^{-2\varepsilon|x|^2}|x|^2)\cdot\nabla u\overline{u}dx\\
	&=4\text{Im}\int_{\mathbb{R}^3}(1-2\varepsilon|x|^2)e^{-2\varepsilon|x|^2}x\cdot\nabla u\overline{u}dx.
\end{align*}
Integrating this identity over $[0,t]$,
\begin{align*}
f_\varepsilon(t)=f_\varepsilon(0)+4\int_0^t\text{Im}\int_{\mathbb{R}^3}\{e^{-\varepsilon|x|^2}(1-2\varepsilon|x|^2)\}e^{-\varepsilon|x|^2}x\cdot\nabla u\overline{u}dxdt.
\end{align*}
Since $e^{-\varepsilon|x|^2}(1-2\varepsilon|x|^2)$ is bounded in $x$ and $\varepsilon$, and $\|e^{-\varepsilon|x|^2}|x|u_0\|_{L^2}\leq\|xu_0\|_{L^2}$, it follows that
\begin{align*}
f_\varepsilon(t)\leq\|xu_0\|_{L^2}^2+C\int_0^t\|\nabla u(s)\|_{L^2}\sqrt{f_\varepsilon(s)}ds.
\end{align*}
This inequality deduces
\begin{align*}
\sqrt{f_\varepsilon(t)}\leq\|xu_0\|_{L^2}+\frac{C}{2}\int_0^t\|\nabla u(s)\|_{L^2}ds
\end{align*}
for any $t\in I$. We take a limit as $\varepsilon\searrow0$ and use Fatou's lemma. Then, we see that $xu(t)\in L^2$ for any $t\in I$.
\end{proof}

\section{Proof of scattering part in Theorem \ref{Scattering versus blowing-up or growing-up}}

In this section, we prove the scattering part in Theorem \ref{Scattering versus blowing-up or growing-up}.

\begin{proposition}[Coercivity \Rnum{1}]\label{Coercivity1}
Let $\frac{7}{3}<p<5$, $V\geq0$, and $V\in\mathcal{K}_0\cap L^\frac{3}{2}$. Assume that $u_0$ satisfies
\begin{align}
M[u_0]^{1-s_c}E_V[u_0]^{s_c}<(1-\delta)M[Q]^{1-s_c}E_0[Q]^{s_c} \label{035}
\end{align}
for some $\delta>0$. Then, there exist $\delta'=\delta'(\delta)>0$, $c=c(\delta,\|u_0\|_{L^2})>0$, and $R=R(\delta,\|u_0\|_{L^2})>0$ such that if $\|u_0\|_{L^2}^{1-s_c}\|\nabla u_0\|_{L^2}^{s_c}<\|Q\|_{L^2}^{1-s_c}\|\nabla Q\|_{L^2}^{s_c}$, then
\begin{itemize}
\item[(i)] $\|u(t)\|_{L^2}^{1-s_c}\|\nabla u(t)\|_{L^2}^{s_c}<(1-2\delta')^\frac{1}{p-1}\|Q\|_{L^2}^{1-s_c}\|\nabla Q\|_{L^2}^{s_c}$,\\
\item[(ii)] $\|\nabla u(t)\|_{L^2}^2-\frac{3(p-1)}{2(p+1)}\|u(t)\|_{L^{p+1}}^{p+1}\geq c\|u(t)\|_{L^{p+1}}^{p+1}$,\\
\item[(iii)] $\|\chi_Ru(t)\|_{L^2}^{1-s_c}\|\nabla(\chi_Ru(t))\|_{L^2}^{s_c}<(1-\delta')^\frac{1}{p-1}\|Q\|_{L^2}^{1-s_c}\|\nabla Q\|_{L^2}^{s_c}$
\end{itemize}
hold, where $\chi_R$ is defined as \eqref{072}.
\end{proposition}

\begin{proof}
First, we will prove (i). By $V\geq0$, Gagliardo-Nirenberg's inequality, and \eqref{021}, we have
\begin{align*}
(1-\delta)^\frac{1}{s_c}M
	&[Q]^\frac{1-s_c}{s_c}E_0[Q]>M[u_0]^\frac{1-s_c}{s_c}E_V[u_0]\\
%	&=\|u(t)\|_{L^2}^\frac{2(1-s_c)}{s_c}\left(\frac{1}{2}\|\nabla u(t)\|_{L^2}^2+\frac{1}{2}\int_{\mathbb{R}^3}V(x)|u(t,x)|^2dx-\frac{1}{p+1}\|u(t)\|_{L^p+1}^{p+1}\right)\\
	&\geq\|u(t)\|_{L^2}^\frac{2(1-s_c)}{s_c}\left(\frac{1}{2}\|\nabla u(t)\|_{L^2}^2-\frac{1}{p+1}C_\text{GN}\|u(t)\|_{L^2}^\frac{5-p}{2}\|\nabla u(t)\|_{L^2}^\frac{3(p-1)}{2}\right)\\
%	&=\frac{1}{2}\|u(t)\|_{L^2}^\frac{2(5-p)}{3p-7}\|\nabla u(t)\|_{L^2}^2-\frac{1}{p+1}C_\text{GN}\|u(t)\|_{L^2}^\frac{3(5-p)(p-1)}{2(3p-7)}\|\nabla u(t)\|_{L^2}^\frac{3(p-1)}{2}\\
	&=\frac{1}{2}\|u(t)\|_{L^2}^\frac{2(5-p)}{3p-7}\|\nabla u(t)\|_{L^2}^2-\frac{2}{3(p-1)}\cdot\frac{\|u(t)\|_{L^2}^\frac{3(5-p)(p-1)}{2(3p-7)}\|\nabla u(t)\|_{L^2}^\frac{3(p-1)}{2}}{\|Q\|_{L^2}^\frac{5-p}{2}\|\nabla Q\|_{L^2}^\frac{3p-7}{2}}
\end{align*}
and henca,
\begin{align*}
(1-\delta)^\frac{1}{s_c}\geq\frac{3(p-1)}{3p-7}\cdot\frac{\|u(t)\|_{L^2}^\frac{2(5-p)}{3p-7}\|\nabla u(t)\|_{L^2}^2}{\|Q\|_{L^2}^\frac{2(5-p)}{3p-7}\|\nabla Q\|_{L^2}^2}-\frac{4}{3p-7}\cdot\frac{\|u(t)\|_{L^2}^\frac{3(5-p)(p-1)}{2(3p-7)}\|\nabla u(t)\|_{L^2}^\frac{3(p-1)}{2}}{\|Q\|_{L^2}^\frac{3(5-p)(p-1)}{2(3p-7)}\|\nabla Q\|_{L^2}^\frac{3(p-1)}{2}}.
\end{align*}
Here, we consider a function $g(y)=\frac{3(p-1)}{3p-7}y^2-\frac{4}{3p-7}y^\frac{3(p-1)}{2}$. Then, $g'(y)=\frac{6(p-1)}{3p-7}y-\frac{6(p-1)}{3p-7}y^\frac{3p-5}{2}$. Solving $g'(y)=0$, we obtain $y=0,1$. We set $y_0=0$ and $y_1=1$. Then, $g$ has a local minimum at $y_0$ and a local maximum at $y_1$. Also, $g(y_1)=1$. Combining these facts and the assumption of Proposition \ref{Coercivity1} (i), there exists $\delta'=\delta'(\delta)>0$ such that
\begin{align*}
\|u(t)\|_{L^2}^\frac{5-p}{3p-7}\|\nabla u(t)\|_{L^2}<(1-2\delta')^\frac{2}{3p-7}\|Q\|_{L^2}^\frac{5-p}{3p-7}\|\nabla Q\|_{L^2},
\end{align*}
%that is,
%\begin{align*}
%\|u(t)\|_{L^2}^{1-s_c}\|\nabla u(t)\|_{L^2}^{s_c}<(1-2\delta')^\frac{1}{p-1}\|Q\|_{L^2}^{1-s_c}\|\nabla Q\|_{L^2}^{s_c},
%\end{align*}
which completes the proof of Proposition \ref{Coercivity1} (i) and implies the uniform estimate \eqref{055} in Theorem \ref{Scattering versus blowing-up or growing-up} (i).\\
Second, we prove (ii). Using Gagliardo-Nirenberg's inequality, this proposition (i), and \eqref{021},
\begin{align*}
E_0[u(t)]
%	&=\frac{1}{2}\|\nabla u(t)\|_{L^2}^2-\frac{1}{p+1}\|u(t)\|_{L^{p+1}}^{p+1}\\
	&\geq\frac{1}{2}\|\nabla u(t)\|_{L^2}^2-\frac{1}{p+1}C_{GN}\|u(t)\|_{L^2}^\frac{5-p}{2}\|\nabla u(t)\|_{L^2}^\frac{3(p-1)}{2}\\
%	&=\|\nabla u(t)\|_{L^2}^2\left(\frac{1}{2}-\frac{1}{p+1}C_{GN}\|u(t)\|_{L^2}^\frac{5-p}{2}\|\nabla u(t)\|_{L^2}^\frac{3p-7}{2}\right)\\
	&>\|\nabla u(t)\|_{L^2}^2\left(\frac{1}{2}-\frac{1}{p+1}C_{GN}(1-2\delta')\|Q\|_{L^2}^\frac{5-p}{2}\|\nabla Q\|_{L^2}^\frac{3p-7}{2}\right)\\
%	&=\|\nabla u(t)\|_{L^2}^2\left(\frac{1}{2}-\frac{2}{3(p-1)}(1-2\delta')\right)\\
	&=\left(\frac{3p-7}{6(p-1)}+\frac{4}{3(p-1)}\delta'\right)\|\nabla u(t)\|_{L^2}^2,
\end{align*}
we have
\begin{align*}
\|\nabla u(t)\|_{L^2}^2-\frac{3(p-1)}{2(p+1)}\|u(t)\|_{L^{p+1}}^{p+1}
%	&=\frac{3(p-1)}{2}\left(\frac{1}{2}\|\nabla u(t)\|_{L^2}^2-\frac{1}{p+1}\|u(t)\|_{L^{p+1}}^{p+1}\right)+\frac{7-3p}{4}\|\nabla u(t)\|_{L^2}^2\\
	&=\frac{3(p-1)}{2}E_0[u(t)]+\frac{7-3p}{4}\|\nabla u(t)\|_{L^2}^2\\
	&>\frac{3(p-1)}{2}\left(\frac{3p-7}{6(p-1)}+\frac{4}{3(p-1)}\delta'\right)\|\nabla u(t)\|_{L^2}^2+\frac{7-3p}{4}\|\nabla u(t)\|_{L^2}^2\\
	&=2\delta'\|\nabla u(t)\|_{L^2}^2>\frac{3(p-1)\delta'}{(p+1)(1-2\delta')}\|u(t)\|_{L^{p+1}}^{p+1},
\end{align*}
which completes the proof of Proposition \ref{Coercivity1} (ii).\\
Finally, we prove (iii). $\|\chi_Ru(t)\|_{L^2}\leq\|u(t)\|_{L^2}$ holds clearly. Since
\begin{align}
\|\nabla(\chi_Ru(t))\|_{L^2}^2
	&=\|\chi_R\nabla u(t)\|_{L^2}^2-\frac{1}{R^2}\int_{\mathbb{R}^3}\chi_R(x)(\Delta\chi)\left(\frac{x}{R}\right)|u(t,x)|^2dx \label{061} \\
	&\leq\|\nabla u(t)\|_{L^2}^2+\mathcal{O}\left(\frac{1}{R^2}M[u_0]\right), \notag
\end{align}
we have
\begin{align*}
\|\chi_Ru(t)\|_{L^2}^{1-s_c}\|\nabla(\chi_Ru(t))\|_{L^2}^{s_c}
	&\leq\|u(t)\|_{L^2}^{1-s_c}\left(\|\nabla u(t)\|_{L^2}^2+\mathcal{O}\left(\frac{1}{R^2}M[u_0]\right)\right)^\frac{s_c}{2}\\
	&\leq\|u(t)\|_{L^2}^{1-s_c}\|\nabla u(t)\|_{L^2}^{s_c}+\mathcal{O}\left(\frac{1}{R^{s_c}}M[u_0]^\frac{1}{2}\right)\\
	&<(1-2\delta')^\frac{1}{p-1}\|Q\|_{L^2}^{1-s_c}\|\nabla Q\|_{L^2}^{s_c}+\mathcal{O}\left(\frac{1}{R^{s_c}}M[u_0]^\frac{1}{2}\right)\\
	&<(1-\delta')^\frac{1}{p-1}\|Q\|_{L^2}^{1-s_c}\|\nabla Q\|_{L^2}^{s_c}
\end{align*}
for sufficiently large $R=R(\delta, \|u_0\|_{L^2})$.
\end{proof}

We define the exponents
\begin{align*}
q_0=\frac{5(p-1)}{2},\ \ r_0=\frac{30(p-1)}{15p-23},\ \ \rho=\frac{5(p-1)}{2p},\ \ \gamma=\frac{30(p-1)}{27p-35}.
\end{align*}
We note that $(q_0,q_0)$ is $\dot{H}^{s_c}$\,admissible, $(q_0,r_0)$ is $L^2$\,admissible, and $(\rho,\gamma)$ is dual $L^2$\,admissible, that is,
\begin{align}
\frac{2}{q_0}+\frac{3}{q_0}=\frac{3}{2}-s_c,\ \ \frac{2}{q_0}+\frac{3}{r_0}=\frac{3}{2},\ \ \frac{2}{\rho'}+\frac{3}{\gamma'}=\frac{3}{2}. \label{057}
\end{align}

\begin{lemma}[Small data global existence, \cite{14}]\label{Small data global existence}
Let $\frac{7}{3}<p<5$, $T>0$, $V\in \mathcal{K}_0\cap L^\frac{3}{2}$, and $\|u(T)\|_{H^{s_c}}\leq A$. There exists $\varepsilon_0>0$ such that, for any $0<\varepsilon<\varepsilon_0$ if
\begin{align*}
\|e^{-i(t-T)\mathcal{H}}u(T)\|_{L^{q_0}(T,\infty;L^{q_0})}<\varepsilon,
\end{align*}
then (NLS$_V$) with initial data $u(T)$ has a unique solution $u$ on $[T,\infty)$ and
\begin{align*}
\|u\|_{L^{q_0}(T,\infty;L^{q_0})}\lesssim\varepsilon.
\end{align*}
\end{lemma}

\begin{proof}
We define a function space
\begin{align*}
E=\left\{u\in CH^{s_c}\cap L^{q_0}W^{s_c,r_0}\!:\!\|u\|_{L^\infty(T,\infty;H^{s_c})}\leq 2CA,\,\|u\|_{L^{q_0}(T,\infty;W^{s_c,r_0})}\leq 2CA,\,\|u\|_{L^{q_0}(T,\infty;L^{q_0})}\leq2\varepsilon\right\}
\end{align*}
and a distance $d$ on $E$
\begin{align*}
d(u_1,u_2)=\|u_1-u_2\|_{L^{q_0}(T,\infty;L^{r_0})}.
\end{align*}
Also, we define a map
\begin{align*}
\Phi(u)(t)=e^{-i(t-T)\mathcal{H}}u(T)+i\int_T^te^{-i(t-s)\mathcal{H}}(|u|^{p-1}u)(s)ds
\end{align*}
for $u\in E$. Using Theorem \ref{Strichartz estimate}, we have
\begin{align*}
\|\Phi(u)\|_{L^{q_0}(T,\infty;L^{q_0})}
	&\leq\|e^{-i(t-T)\mathcal{H}}u(T)\|_{L^{q_0}(T,\infty;L^{q_0})}+\left\|\int_T^te^{-i(t-s)\mathcal{H}}(|u|^{p-1}u)(s)ds\right\|_{L^{q_0}(T,\infty;L^{q_0})}\\
%	&\lesssim\|e^{-i(t-T)\mathcal{H}}u(T)\|_{L^{q_0}(T,\infty;L^{q_0})}+\left\|\int_T^te^{-i(t-s)\mathcal{H}}|\nabla|^{s_c}(|u|^{p-1}u)(s)ds\right\|_{L^{q_0}(T,\infty;L^{r_0})}\\
	&\leq\varepsilon+C\||u|^{p-1}u\|_{L^\rho(T,\infty;W^{s_c,\gamma})}\\
	&\leq\varepsilon+C\|u\|_{L^{q_0}(T,\infty;L^{q_0})}^{p-1}\|u\|_{L^{q_0}(T,\infty;W^{s_c,r_0})}\\
%	&\leq\varepsilon+C2^pA\varepsilon^{p-1}\\
	&\leq(1+2^pCA\varepsilon^{p-2})\varepsilon
\end{align*}
and
\begin{align*}
\|\Phi(u)\|_{L^\infty(T,\infty;H^{s_c})}
	&\leq\left\|e^{-i(t-T)\mathcal{H}}u(T)\right\|_{L^\infty(T,\infty;H^{s_c})}+\left\|\int_T^te^{-i(t-s)\mathcal{H}}(|u|^{p-1}u)(s)ds\right\|_{L^\infty(T,\infty;H^{s_c})}\\
	&\leq C\|u(T)\|_{H^{s_c}}+C\||u|^{p-1}u\|_{L^\rho(T,\infty;W^{s_c,\gamma})}\\
	&\leq C\|u(T)\|_{H^{s_c}}+C\|u\|_{L^{q_0}(T,\infty;L^{q_0})}^{p-1}\|u\|_{L^{q_0}(T,\infty;W^{s_c,r_0})}\\
%	&\leq CA+2^pCA\varepsilon^{p-1}\\
	&\leq(1+2^pC\varepsilon^{p-1})CA.
\end{align*}
Similarly, we have
\begin{align*}
\|\Phi(u)\|_{L^{q_0}(T,\infty;W^{s_c,r_0})}\leq(1+2^pC\varepsilon^{p-1})CA.
\end{align*}
Thus, if $\varepsilon>0$ satisfies $\max\{2^pCA\varepsilon^{p-2},2^pC\varepsilon^{p-1}\}\leq\frac{1}{2}$, then
\begin{align*}
\|u\|_{L^\infty(T,\infty;H^{s_c})}\leq 2CA,\ \ \ \|u\|_{L^{q_0}(T,\infty;W^{s_c,r_0})}\leq 2CA,\ \ \ \|u\|_{L^{q_0}(T,\infty;L^{q_0})}\leq2\varepsilon.
\end{align*}
Also, for $u, v\in E$,
\begin{align*}
\|\Phi(u)-\Phi(v)\|_{L^{q_0}(T,\infty;L^{r_0})}
	&\leq\left\|\int_{T}^te^{-i(t-s)\mathcal{H}}(|u|^{p-1}u-|v|^{p-1}v)(s)ds\right\|_{L^{q_0}(T,\infty;L^{r_0})}\\
	&\hspace{0.0cm}\leq C\||u|^{p-1}u-|v|^{p-1}v\|_{L^\rho(T,\infty;L^\gamma)}\\
%	&\hspace{0.5cm}\leq C\|(|u|^{p-1}+|v|^{p-1})(u-v)\|_{L^\rho(T,\infty;L^\gamma)}\\
%	&\hspace{0.5cm}\leq C\left(\||u|^{p-1}\|_{L^{\frac{q_0}{p-1}}(T,\infty;L^\frac{q_0}{p-1})}+\||v|^{p-1}\|_{L^{\frac{q_0}{p-1}}(T,\infty;L^\frac{q_0}{p-1})}\right)\|u-v\|_{L^{q_0}(T,\infty;L^{r_0})}\\
	&\hspace{0.0cm}\leq C\left(\|u\|_{L^{q_0}(T,\infty;L^{q_0})}^{p-1}+\|v\|_{L^{q_0}(T,\infty;L^{q_0})}^{p-1}\right)\|u-v\|_{L^{q_0}(T,\infty;L^{r_0})}\\
	&\hspace{0.0cm}\leq 2^pC\varepsilon^{p-1}\|u-v\|_{L^{q_0}(T,\infty;L^{r_0})}\\
	&\hspace{0.0cm}\leq \frac{1}{2}\|u-v\|_{L^{q_0}(T,\infty;L^{r_0})}
\end{align*}
Therefore, $\Phi$ is a contraction map on $E$, and hence, there exists a unique solution $u$ to (NLS$_V$) on $[T,\infty)$.
\end{proof}

\begin{lemma}[Small data scattering]\label{Small data scattering}
Let $\frac{7}{3}<p<5$, $T>0$, and $V\in\mathcal{K}_0\cap L^\frac{3}{2}$. $u\in L^\infty H^1$ is a time global solution to (NLS$_V$) satisfying
\begin{align*}
\|u\|_{L^\infty H^1}\leq E.
\end{align*}
Then, there exists $\varepsilon>0$ such that if
\begin{align*}
\|e^{-i(t-T)\mathcal{H}}u(T)\|_{L^{q_0}(T,\infty;L^{q_0})}<\varepsilon,
\end{align*}
then $u$ scatters in positive time.
\end{lemma}

\begin{proof}
We take $\varepsilon>0$ as in Lemma \ref{Small data global existence} with $A=E$. From Lemma \ref{Small data global existence}, the unique solution $u$ to (NLS$_V$) satisfies
\begin{align*}
\|u\|_{L^{q_0}(T,\infty;W^{s_c,r_0})}\leq 2CE\ \ \text{ and }\ \ \|u\|_{L^{q_0}(T,\infty;L^{q_0})}\leq 2\varepsilon.
\end{align*}
Here, we take exponents $q_1$, $r_1$, $q_2$, $r_2$, and $r$ as follows.\\
Case 1: $\frac{7}{3}<p\leq 3$.\\
We choose
\begin{align*}
q_1:=2(p-1)^+,\ \ r_1:=\frac{6(p-1)}{3p-5}^-,\ \ q_2:=\infty^-,\ \ r_2:=2^+,\ \ r:=3(p-1)^-
\end{align*}
satisfying $(q_1,r_1)$ and $(q_2,r_2)$ are $L^2$\,admissible pairs, the embedding $\dot{W}^{s_c,r_1}\hookrightarrow L^r$ holds, $\dot{W}_V^{s_c,r_1}$ and $\dot{W}^{s_c,r_1}$ are equivalent, and $\dot{W}_V^{1,r_2}$ and $\dot{W}^{1,r_2}$ are equivalent.\\
Case 2: $3<p<5$.
\begin{align*}
q_1:=\frac{4(p-1)^2}{p+1},\ \ r_1:=\frac{6(p-1)^2}{3p^2-7p+2},\ \ q_2:=\frac{4(p-1)}{p-3},\ \ r_2:=\frac{3(p-1)}{p},\ \ r:=\frac{6(p-1)^2}{3p-5}.
\end{align*}
Then, $(q_1,r_1)$ and $(q_2,r_2)$ are $L^2$\,admissible pairs, the embedding $\dot{W}^{s_c,r_1}\hookrightarrow L^r$ holds, $\dot{W}_V^{s_c,r_1}$ and $\dot{W}^{s_c,r_1}$ are equivalent, and $\dot{W}_V^{1,r_2}$ and $\dot{W}^{1,r_2}$ are equivalent. Then,
\begin{align*}
\|u\|_{L^{q_1}(T,\infty;W^{s_c,r_1})}\leq C\|u(T)\|_{H^{s_c}}+C\|u\|_{L^{q_0}(T,\infty;L^{q_0})}^{p-1}\|u\|_{L^{q_0}(T,\infty;W^{s_c,r_0})}<\infty.
\end{align*}
Thus, we have
\begin{align*}
u\in L^\infty(T,\infty;H^1)\cap L^{q_0}(T,\infty;W^{s_c,r_0})\cap L^{q_1}(T,\infty;W^{s_c,r_1}).
\end{align*}
The following estimate
\begin{align*}
\|u\|_{L^{q_2}(T,\infty;W^{1,r_2})}
	&\leq c\|u(T)\|_{H^1}+c\||u|^{p-1}u\|_{L^2(T,\infty;W^{1,\frac{6}{5}})}\\
	&\leq c\|u(T)\|_{H^1}+c\|u\|_{L^{q_1}(T,\infty;L^r)}^{p-1}\|u\|_{L^{q_2}(T,\infty;W^{1,r_2})}\\
	&\leq c\|u(T)\|_{H^1}+c\|u\|_{L^{q_1}(T,\infty;\dot{W}^{s_c,r_1})}^{p-1}\|u\|_{L^{q_2}(T,\infty;W^{1,r_2})}
\end{align*}
deduces $u\in L^{q_2}(T,\infty;W^{1,r_2})$ and hence, we obtain
\begin{align*}
\|e^{it\mathcal{H}}u(t)-e^{i\tau\mathcal{H}}u(\tau)\|_{H^1}
	&=\left\|\int_\tau^te^{is\mathcal{H}}(|u|^{p-1}u)(s)ds\right\|_{H^1}\leq c\||u|^{p-1}u\|_{L^2(\tau,t;W^{1,\frac{6}{5}})}\\
	&\leq c\|u\|_{L^{q_1}(\tau,t;W^{s_c,r_1})}^{p-1}\|u\|_{L^{q_2}(\tau,t;W^{1,r_2})}\longrightarrow0\ \ \text{ as }\ \ t>\tau\rightarrow\infty.
\end{align*}
Therefore, $\{e^{it\mathcal{H}}u(t)\}$ is a Cauchy sequence in $H^1$.
\end{proof}

\begin{theorem}[Scattering criterion, \cite{18}]\label{Scattering criterion}
Let $E>0$, $\frac{7}{3}<p<5$, $V\geq0$, $V\in\mathcal{K}_0\cap L^\frac{3}{2}$, and $V$ is radially symmetric. Suppose that $u:\mathbb{R}\times\mathbb{R}^3\longrightarrow\mathbb{C}$ is radially symmetric and a solution to (NLS$_V$) satisfying
\begin{align*}
\|u\|_{L^\infty H^1}\leq E.
\end{align*}
Then, there exist $\varepsilon=\varepsilon(E)>0$ and $R=R(E)>0$ such that if
\begin{align*}
\liminf_{t\rightarrow\infty}\int_{|x|\leq R}|u(t,x)|^2dx\leq\varepsilon^2,
\end{align*}
then $u$ scatters in positive time.
\end{theorem}

The proof of this theorem is based on the argument in \cite[{\S}4]{18}. We have to change exponents of function spaces.

\begin{proof}
Set $0<\varepsilon<1$ and $R>0$, which will be chosen later. Using Theorem \ref{Strichartz estimate},
\begin{align*}
\|e^{-it\mathcal{H}}u_0\|_{L^{q_0}L^{q_0}}\leq c\|u_0\|_{\dot{H}^{s_c}}<\infty.
\end{align*}
Thus, there exists $T_0>\varepsilon^{-1}$ such that
\begin{align}
\|e^{-it\mathcal{H}}u_0\|_{L^{q_0}(T_0,\infty;L^{q_0})}<\varepsilon. \label{036}
\end{align}
By the assumption of Theorem \ref{Scattering criterion}, there exists $T>T_0$ such that
\begin{align}
\int_{|x|\leq R}|u(T,x)|^2dx\leq2\varepsilon^2.\label{073}
\end{align}
Since $u$ satisfies the integral equation,
\begin{align*}
u(T)=e^{-iT\mathcal{H}}u_0+i\int_0^Te^{-i(T-s)\mathcal{H}}(|u|^{p-1}u)(s)ds,
\end{align*}
we have
\begin{align}
e^{-i(t-T)\mathcal{H}}u(T)
	&=e^{-it\mathcal{H}}u_0+i\int_0^Te^{-i(t-s)\mathcal{H}}(|u|^{p-1}u)(s)ds \notag \\
	&=e^{-it\mathcal{H}}u_0+i\int_{I_1}e^{-i(t-s)\mathcal{H}}(|u|^{p-1}u)(s)ds+i\int_{I_2}e^{-i(t-s)\mathcal{H}}(|u|^{p-1}u)(s)ds \notag \\
	&=:e^{-it\mathcal{H}}u_0+F_1(t)+F_2(t) \label{060}
\end{align}
where $I_1:=[0,T-\varepsilon^{-\theta}]$ and $I_2=[T-\varepsilon^{-\theta},T]$. Here, we will choose $0<\theta=\theta(p)<1$ later. First, we consider estimating $\|F_1\|_{L^{q_0}(T,\infty;L^{q_0})}$. By the integral equation, we have
\begin{align*}
u(T-\varepsilon^{-\theta})=e^{-i(T-\varepsilon^{-\theta})\mathcal{H}}u_0+i\int_{I_1}e^{-i(T-\varepsilon^{-\theta}-s)\mathcal{H}}(|u|^{p-1}u)(s)ds.
\end{align*}
Operating $e^{-i(t-T+\varepsilon^{-\theta})\mathcal{H}}$ to this identity, we have
\begin{align*}
e^{-i(t-T+\varepsilon^{-\theta})\mathcal{H}}u(T-\varepsilon^{-\theta})=e^{-it\mathcal{H}}u_0+i\int_{I_1}e^{-i(t-s)\mathcal{H}}(|u|^{p-1}u)(s)ds.
\end{align*}
Hence,
\begin{align*}
F_1(t)=e^{-i(t-T+\varepsilon^{-\theta})\mathcal{H}}u(T-\varepsilon^{-\theta})-e^{-it\mathcal{H}}u_0.
\end{align*}
We take $\mu$ satisfying
\begin{align*}
\frac{3p-7}{3(p-1)}<\mu<\min\left\{\frac{3p-7}{p-1},\frac{5p-9}{5(p-1)}\right\}.
\end{align*}
We set
\begin{align*}
q_3=\frac{20(p-1)}{15(1-\mu)p+15\mu-27},\ \ r_3=\frac{4(p-1)}{-3(1-\mu)p-3\mu+7}.
\end{align*}
Then, the following relation holds:
\begin{align*}
\frac{1}{q_0}=\frac{1}{q_3}+\frac{1}{r_3},\ \ \frac{2}{q_3(1-\mu)}+\frac{3}{q_0(1-\mu)}=\frac{3}{2},\ \ r_3\mu>2,\ \ q_3(1-\mu)>2.
\end{align*}
Theorem \ref{Strichartz estimate} implies
\begin{align}
\|F_1\|_{L^{q_3(1-\mu)}(T,\infty;L^{q_0(1-\mu)})}\lesssim\|u(T-\varepsilon^{-\theta})\|_{L^2}+\|u_0\|_{L^2}=2\|u_0\|_{L^2}\lesssim1. \label{037}
\end{align}
On the other hand, by Theorem \ref{Dispersive estimate} and Sobolev's embedding, we have
\begin{align*}
\|F_1(t)\|_{L^\infty}
	&\leq\int_{I_1}\|e^{i(t-s)\mathcal{H}}(|u|^{p-1}u)(s)\|_{L^\infty}ds\lesssim\int_{I_1}\frac{1}{|t-s|^\frac{3}{2}}\|u(s)\|_{L^p}^pds\\
	&\lesssim\|u\|_{L^\infty H^1}^p\int_0^{T-\varepsilon^{-\theta}}(t-s)^{-\frac{3}{2}}ds\lesssim_E\left[(t-s)^{-\frac{1}{2}}\right]_0^{T-\varepsilon^{-\theta}}\lesssim(t-T+\varepsilon^{-\theta})^{-\frac{1}{2}}.
\end{align*}
Thus, we have
\begin{align}
\|F_1\|_{L^{r_3\mu}(T,\infty;L^\infty)}
	&\lesssim\left\|(t-T+\varepsilon^{-\theta})^{-\frac{1}{2}}\right\|_{L^{r_3\mu}(T,\infty)}=\left(\int_T^\infty(t-T+\varepsilon^{-\theta})^{-\frac{1}{2}r_3\mu}dt\right)^\frac{1}{r_3\mu} \notag \\
	&\sim\left(\left[-(t-T+\varepsilon^{-\theta})^{1-\frac{1}{2}r_3\mu}\right]_T^\infty\right)^\frac{1}{r_3\mu}=\varepsilon^{\left(\frac{1}{2}-\frac{1}{r_3\mu}\right)\theta}. \label{038}
\end{align}
Combining these inequalities \eqref{037} and \eqref{038}, we have
\begin{align}
\|F_1\|_{L^{q_0}(T,\infty;L^{q_0})}\leq\|F_1\|_{L^{q_3(1-\mu)}(T,\infty;L^{q_0(1-\mu)})}^{1-\mu}\|F_1\|_{L^{r_3\mu}(T,\infty;L^\infty)}^\mu\lesssim\varepsilon^{\left(\frac{\mu}{2}-\frac{1}{r_3}\right)\theta}. \label{039}
\end{align}
Next, we consider estimating $\|F_2\|_{L^{q_0}(T,\infty;L^{q_0})}$. Applying Proposition \ref{Virial identity} and the assumption of Theorem \ref{Scattering criterion}, we have
\begin{align*}
\left|\frac{d}{dt}\int_{\mathbb{R}^3}\chi_R(x)|u(t,x)|^2dx\right|&\leq2\left|\int_{\mathbb{R}^3}\overline{u}\nabla{u}\cdot\nabla\chi_Rdx\right|\leq2\|\nabla\chi_R\|_{L^\infty}\|u(t)\|_{L^2}\|\nabla u(t)\|_{L^2}\leq\frac{c}{R},
\end{align*}
where $\chi_R$ is defined as \eqref{072}. Thus, we have
\begin{align*}
-\frac{c}{R}\leq\frac{d}{dt}\int_{\mathbb{R}^3}\chi_R(x)|u(t,x)|^2dx\leq\frac{c}{R}.
\end{align*}
Integrating each terms in this inequality over $[t,T]$,
\begin{align*}
-\frac{c}{R}(T-t)\leq\int_{\mathbb{R}^3}\chi_R(x)|u(T,x)|^2dx-\int_{\mathbb{R}^3}\chi_R(x)|u(t,x)|dx\leq\frac{c}{R}(T-t).
\end{align*}
The left inequality implies
\begin{align}
\int_{\mathbb{R}^3}\chi_R(x)|u(t,x)|^2dx\leq\int_{\mathbb{R}^3}\chi_R(x)|u(T,x)|^2dx+\frac{c}{R}(T-t). \label{040}
\end{align}
Here, we choose $R>0$ satisfying $R>\varepsilon^{-2-\theta}$. By taking supremum on $I_2$ for \eqref{040} and using \eqref{073}, we have
\begin{align*}
\sup_{t\in I_2}\int_{\mathbb{R}^3}\chi_R(x)|u(t,x)|^2dx\leq\varepsilon^2+c\varepsilon^{2+\theta}\varepsilon^{-\theta}\lesssim c\varepsilon^2,
\end{align*}
that is,
\begin{align}
\|\chi_Ru\|_{L^\infty(I_2;L^2)}\leq c\varepsilon. \label{041}
\end{align}
By Lemma \ref{Radial Sobolev embedding}, this estimate \eqref{041}, H\"older's inequality, and Sobolev's embedding,
\begin{align}
	&\|u\|_{L^\frac{10}{3}(I_2;L^\frac{10}{3})}
%	=\left\|1\cdot\|u\|_{L^\frac{10}{3}}\right\|_{L^\frac{10}{3}(I_2)}
	\leq\|1\|_{L^\frac{10}{3}(I_2)}\|u\|_{L^\infty(I_2;L^\frac{10}{3})}\leq\varepsilon^{-\frac{3}{10}\theta}\|u\|_{L^\infty(I_2;L^\frac{10}{3})} \notag \\
	&\hspace{1.1cm}\leq\varepsilon^{-\frac{3}{10}\theta}\left(\|\chi_Ru\|_{L^\infty(I_2;L^\frac{10}{3})}+\|(1-\chi_R)u\|_{L^\infty(I_2;L^\frac{10}{3})}\right) \notag \\
%	&\hspace{1.1cm}\leq\varepsilon^{-\frac{3}{10}\theta}\left(\|\chi_Ru^\frac{2}{5}\|_{L^\infty(I_2;L^5)}\|u^\frac{3}{5}\|_{L^\infty(I_2;L^{10})}+\|(1-\chi_R)u^\frac{2}{5}\|_{L^\infty(I_2;L^\infty)}\|u^\frac{3}{5}\|_{L^\infty(I_2;L^\frac{10}{3})}\right) \notag \\
	&\hspace{1.1cm}\leq\varepsilon^{-\frac{3}{10}\theta}\left(\|\chi_Ru\|_{L^\infty(I_2;L^2)}^\frac{2}{5}\|u\|_{L^\infty(I_2;L^{6})}^\frac{3}{5}+\|(1-\chi_R)u\|_{L^\infty(I_2;L^\infty)}^\frac{2}{5}\|u\|_{L^\infty(I_2;L^2)}^\frac{3}{5}\right) \notag \\
	&\hspace{1.1cm}\leq\varepsilon^{-\frac{3}{10}\theta}\left(c\,\varepsilon^\frac{2}{5}\|u\|_{L^\infty(I_2;\dot{H}^1)}^\frac{3}{5}+\||x|^{-\frac{1}{2}}|x|^\frac{1}{2}u\|_{L^\infty(I_2;L^\infty(|x|\geq R/2))}^\frac{2}{5}\|u\|_{L^\infty(I_2;L^2)}^\frac{3}{5}\right) \notag \\
	&\hspace{1.1cm}\leq\varepsilon^{-\frac{3}{10}\theta}\left(c\,\varepsilon^\frac{2}{5}\|u\|_{L^\infty(I_2;\dot{H}^1)}^\frac{3}{5}+R^{-\frac{1}{5}}\|u\|_{L^\infty(I_2;H^1)}^\frac{2}{5}\|u\|_{L^\infty(I_2;L^2)}^\frac{3}{5}\right) \notag \\
	&\hspace{1.1cm}\leq\varepsilon^{-\frac{3}{10}\theta}\left(c\,\varepsilon^\frac{2}{5}+c\varepsilon^{\frac{1}{5}(2+\theta)}\right)\lesssim c\,\varepsilon^{\frac{2}{5}-\frac{3}{10}\theta}. \label{042}
\end{align}
By using Theorem \ref{Strichartz estimate} and a continuity argument, we have
\begin{align}
\||\nabla|^{s_c}u\|_{L^{10}(I_2;L^\frac{30}{13})}^{10}+\|u\|_{L^{10}(I_2;L^{10})}^{10}\lesssim1+|I_2|. \label{043}
\end{align}
From this inequality, Sobolev's embedding, Theorem \ref{Strichartz estimate}, Lemma \ref{Fractional calculus}, and H\"older's inequality, it follows that
\begin{align*}
\|F_2\|_{L^{q_0}(T,\infty;L^{q_0})}
%	&\lesssim\||\nabla|^{s_c}F_2\|_{L^{q_0}(T,\infty;L^{r_0})}\\
%	&\lesssim\left\|\int_{I_2}e^{-i(t-s)\mathcal{H}}|\nabla|^{s_c}(|u|^{p-1}u)(s)ds\right\|_{L^{q_0}(T,\infty;L^{r_0})}\\
	&\lesssim\||\nabla|^{s_c}(|u|^{p-1}u)\|_{L^2(I_2;L^\frac{6}{5})}\lesssim\left\|\||u|^{p-1}\|_{L^\frac{5}{2}}\||\nabla|^{s_c}u\|_{L^\frac{30}{13}}\right\|_{L^2(I_2)}\\
	&\leq\||u|^{p-1}\|_{L^\frac{5}{2}(I_2;L^\frac{5}{2})}\||\nabla|^{s_c}u\|_{L^{10}(I_2;L^\frac{30}{13})}\\
%	&=\|u\|_{L^{\frac{5}{2}(p-1)}(I_2;L^{\frac{5}{2}(p-1)})}^{p-1}\||\nabla|^{s_c}u\|_{L^{10}(I_2;L^\frac{30}{13})}\\
	&\lesssim(1+|I_2|)^\frac{1}{10}\left(\|u\|_{L^\frac{10}{3}(I_2;L^\frac{10}{3})}^{1-s_c}\|u\|_{L^{10}(I_2;L^{10})}^{s_c}\right)^{p-1}\\
	&\lesssim|I_2|^\frac{1}{10}\left(\varepsilon^{\left(\frac{2}{5}-\frac{3}{10}\theta\right)(1-s_c)}|I_2|^{\frac{1}{10}s_c}\right)^{p-1}\\
%	=|I_2|^{\frac{1}{10}+\frac{1}{10}s_c(p-1)}\varepsilon^{\left(\frac{2}{5}-\frac{3}{10}\theta\right)(1-s_c)(p-1)}\\
	&=\varepsilon^{-\frac{1}{10}\theta-\frac{1}{10}s_c(p-1)\theta}\varepsilon^{\left(\frac{2}{5}-\frac{3}{10}\theta\right)(1-s_c)(p-1)}=\varepsilon^{\frac{5-p}{5}-\frac{1}{2}\theta}.
%	=:\varepsilon^\omega.
\end{align*}
%Here, we calculate a power of $\varepsilon$.
%\begin{align*}
%\omega
%	&=-\frac{1}{10}\theta-\frac{1}{10}s_c(p-1)\theta+\left(\frac{2}{5}-\frac{3}{10}\theta\right)(1-s_c)(p-1)\\
%	&=-\frac{1}{10}\theta-\frac{1}{10}\left(\frac{3}{2}-\frac{2}{p-1}\right)(p-1)\theta+\left(\frac{2}{5}-\frac{3}{10}\theta\right)\left(-\frac{1}{2}+\frac{2}{p-1}\right)(p-1)\\
%	&=-\frac{1}{10}\theta-\frac{3p-7}{20}\theta+\frac{5-p}{5}-\frac{3(5-p)}{20}\theta\\
%	&=\frac{5-p}{5}-\frac{1}{2}\theta.
%\end{align*}
Thus, if we take $\theta=\frac{5-p}{5}\in(0,1)$, then
\begin{align}
\|F_2\|_{L^{q_0}(T,\infty;L^{q_0})}\lesssim\varepsilon^{\frac{1}{2}\theta}. \label{044}
\end{align}
Combining \eqref{036}, \eqref{060}, \eqref{039}, and \eqref{044}, we obtain
\begin{align*}
\|e^{-i(t-T)\mathcal{H}}u(T)\|_{L^{q_0}(T,\infty;L^{q_0})}\lesssim\varepsilon+\varepsilon^{\frac{1}{2}\theta}.
\end{align*}
From Lemma \ref{Small data global existence} and Lemma \ref{Small data scattering}, $u$ scatters.
\end{proof}

\begin{proposition}[Virial/Morawetz estimate]\label{Virial/Morawetz estimate}
Let $\frac{7}{3}<p<5$, $T>0$, $V\geq0$, $V\in\mathcal{K}_0\cap L^\frac{3}{2}$, $x\cdot\nabla V\leq0$, $x\cdot\nabla V\in L^\frac{3}{2}$, and $V$ be radially symmetric. We assume that $u$ is a global solution to (NLS$_V$) with radial symmetry satisfying
\begin{align*}
M[u_0]^{1-s_c}E_V[u_0]^{s_c}<(1-\delta)M[Q]^{1-s_c}E_0[Q]^{s_c},
\end{align*}
\begin{align*}
\|u_0\|_{L^2}^{1-s_c}\|\nabla u_0\|_{L^2}^{s_c}<\|Q\|_{L^2}^{1-s_c}\|\nabla Q\|_{L^2}^{s_c}
\end{align*}
for some $\delta>0$. Then, it follows that
\begin{align*}
\frac{1}{T}\int_0^T\int_{|x|\leq\frac{R}{2}}|u(t,x)|^{p+1}dxdt\lesssim\frac{R}{T}+\frac{1}{R^2}+\frac{1}{R^{p-1}}+o_R(1)
\end{align*}
for sufficiently large $R=R(\delta,M[u_0],Q)$.
\end{proposition}

\begin{proof}
We set a function
\begin{equation}
\notag w(x)=
\begin{cases}
&\hspace{-0.15cm}|x|^2\hspace{0.98cm}\left(|x|\leq1\right),\\
&\hspace{-0.4cm}\text{smooth}\hspace{0.6cm}\left(1<|x|<2\right),\\
&\hspace{-0.4cm}3|x|-4\hspace{0.55cm}\left(2\leq|x|\right),
\end{cases}
\end{equation}
which satisfies $\partial_rw\geq0$, $\partial_r^2w\geq0$, and $|\partial^\alpha w(x)|\lesssim_\alpha |x|^{-|\alpha|+1}$ for $1<|x|<2$. We define $w_R$ as
\begin{align}
w_R(x)=R^2w\left(\frac{x}{R}\right)\label{075}
\end{align}
for $R>0$. By a direct calculation, we have $\partial_jw_R=2x_j$, $\partial_{kj}w_R=2\delta_{kj}$, $\Delta w_R=6$, and $\Delta\Delta w_R=0$ for $|x|\leq R$ and $\partial_jw_R=\frac{3Rx_j}{|x|}$, $\partial_{kj}w_R=\frac{3R}{|x|}\left[\delta_{jk}-\frac{x_jx_k}{|x|^2}\right]$, $\Delta w_R=\frac{6R}{|x|}$, and $\Delta\Delta w_R=0$ for $2R\leq|x|$.
We difine a function $M(t)$ as
\begin{align*}
M(t)\vcentcolon=2\text{Im}\int_{\mathbb{R}^3}\overline{u}\nabla u\cdot\nabla w_Rdx.
\end{align*}
By H\"older's inequlity, we have
\begin{align*}
|M(t)|
%	&=\left|2\text{Im}\int_{\mathbb{R}^3}\overline{u}\nabla u\cdot\nabla w_Rdx\right|\\
	&\lesssim\|u(t)\|_{L^2}\|\nabla u(t)\|_{L^2}\left(\|\nabla w_R\|_{L^\infty(|x|\leq R)}+\|\nabla w_R\|_{L^\infty(R\leq|x|\leq 2R)}+\|\nabla w_R\|_{L^\infty(2R\leq|x|)}\right)\\
	&\lesssim\|u(t)\|_{L^2}\|\nabla u(t)\|_{L^2}R.
\end{align*}
Hence, we have
\begin{align}
\sup_{t\in\mathbb{R}}|M(t)|\lesssim_uR. \label{074}
\end{align}
Using Proposition \ref{Virial identity},
\begin{align}
\frac{d}{dt}
	&M(t)=\int_{|x|\leq R}8|\nabla u|^2-\frac{12(p-1)}{p+1}|u|^{p+1}-4(x\cdot\nabla V)|u|^2dx \label{045} \\
	&\hspace{0.5cm}+\int_{R\leq|x|\leq 2R}4\text{Re}\,\partial_{jk}w_Ru_j\overline{u}_k-\frac{2(p-1)}{p+1}\Delta w_R|u|^{p+1}-\Delta^2w_R|u|^2-2(\nabla w_R\cdot\nabla V)|u|^2dx \label{046} \\
	&\hspace{0.5cm}+\int_{2R\leq|x|}12\text{Re}\frac{R}{|x|}\left[\delta_{jk}-\frac{x_jx_k}{|x|^2}\right]\overline{u_j}u_k-\frac{12(p-1)}{p+1}\frac{R}{|x|}|u|^{p+1}-\frac{6R}{|x|}(x\cdot\nabla V)|u|^2dx. \label{047}
\end{align}
For \eqref{045}, using \eqref{061} and Proposition \ref{Coercivity1} (ii), (iii), we have
\begin{align}
	&\int_{|x|\leq R}8|\nabla u|^2-\frac{12(p-1)}{p+1}|u|^{p+1}-4(x\cdot\nabla V)|u|^2dx \notag \\
	&\hspace{1cm}\geq\int_{|x|\leq R}8|\chi_R\nabla u|^2-\frac{12(p-1)}{p+1}|\chi_Ru|^{p+1}dx+\frac{12(p-1)}{p+1}\int_{|x|\leq R}|\chi_Ru|^{p+1}-|u|^{p+1}dx \notag \\
	&\hspace{1cm}=\int_{\mathbb{R}^3}8|\chi_R\nabla u|^2-\frac{12(p-1)}{p+1}|\chi_Ru|^{p+1}dx+\frac{12(p-1)}{p+1}\int_{\frac{R}{2}\leq|x|\leq R}|\chi_Ru|^{p+1}-|u|^{p+1}dx \notag \\
	&\hspace{1cm}=\int_{\mathbb{R}^3}8|\nabla(\chi_Ru)|^2+\frac{8}{R^2}\chi_R(\Delta\chi)\left(\frac{x}{R}\right)|u|^2-\frac{12(p-1)}{p+1}|\chi_Ru|^{p+1}dx \notag \\
	&\hspace{7.5cm}+\frac{12(p-1)}{p+1}\int_{\frac{R}{2}\leq|x|\leq R}|\chi_Ru|^{p+1}-|u|^{p+1}dx \notag \\
	&\hspace{1cm}\geq c\|\chi_Ru\|_{L^{p+1}}^{p+1}+\int_{\mathbb{R}^3}\frac{8}{R^2}\chi_R(\Delta\chi)\left(\frac{x}{R}\right)|u|^2dx-\frac{12(p-1)}{p+1}\int_{\frac{R}{2}\leq|x|\leq R}|u|^{p+1}dx, \label{048}
\end{align}
where $\chi_R$ is defined as \eqref{072}. For \eqref{046}, by the identity
\begin{align*}
\partial_{jk}w_R=w_R''(r)\frac{x_jx_k}{r^2}+w_R'(r)\left(\frac{\delta_{jk}}{r}-\frac{x_jx_k}{r^3}\right),
\end{align*}
we have
\begin{align}
\sum_{1\leq j,k\leq3}\overline{u}_ju_k\partial_{jk}w_R
	&=\sum_{j=1}^3|u_j|^2\partial_j^2w_R+\sum_{1\leq j\neq k\leq3}\overline{u}_ju_k\partial_{jk}w_R \notag \\
	&=\sum_{j=1}^3|u'(r)|^2\frac{x_j^2}{r^2}\cdot\left\{w_R''(r)\frac{x_j^2}{r^2}+w_R'(r)\left(\frac{1}{r}-\frac{x_j^2}{r^3}\right)\right\} \notag \\
	&\hspace{2.0cm}+\sum_{1\leq j\neq k\leq3}|u'(r)|^2\frac{x_jx_k}{r^2}\cdot\left(w_R''(r)\frac{x_jx_k}{r^2}-w_R'(r)\frac{x_jx_k}{r^3}\right) \notag \\
	&=\sum_{1\leq j,k\leq3}|u'(r)|^2w_R''(r)\frac{x_j^2x_k^2}{r^4} \notag \\
	&\hspace{2.0cm}+\sum_{j=1}^3|u'(r)|^2w_R'(r)\frac{x_j^2(r^2-x_j^2)}{r^5}-\sum_{1\leq j\neq k\leq3}|u'(r)|^2w_R'(r)\frac{x_j^2x_k^2}{r^5} \notag \\
	&=\sum_{1\leq j,k\leq3}|u'(r)|^2w_R''(r)\frac{x_j^2x_k^2}{r^4}\geq0. \label{049}
\end{align}
Since $x\cdot\nabla V\in L^\frac{3}{2}$, we have
\begin{align}
\sup_{t\in\mathbb{R}}\left|\int_{R\leq|x|\leq2R}(\nabla w_R\cdot\nabla V)|u|^2dx\right|
	&\lesssim\sup_{t\in\mathbb{R}}\int_{R\leq|x|\leq2R}|x\cdot\nabla V||u|^2dx \notag \\
	&\leq\|x\cdot\nabla V\|_{L^\frac{3}{2}(R\leq|x|\leq2R)}\|u\|_{L^\infty L^6}^2 \notag \\
	&\lesssim\|x\cdot\nabla V\|_{L^\frac{3}{2}(R\leq|x|\leq2R)}\|\nabla u\|_{L^\infty L^2}^2 \notag \\
	&\leq\|x\cdot\nabla V\|_{L^\frac{3}{2}(R\leq|x|\leq2R)}\frac{\|Q\|_{L^2}^2\|\nabla Q\|_{L^2}^2}{\|u\|_{L^2}^2} \notag \\
	&=o_R(1). \label{050}
\end{align}
For \eqref{047}, since a function $u$ is radially symmetric,
\begin{align}
\left(\delta_{jk}-\frac{x_jx_k}{|x|^2}\right)\overline{u_j}u_k
	&=\sum_{1\leq j,k\leq3}\left(\delta_{jk}-\frac{x_jx_k}{|x|^2}\right)\overline{u_j}u_k \notag \\
	&=\sum_{j=1}^3\left(1-\frac{x_j^2}{|x|^2}\right)\cdot|u'(r)|^2\frac{x_j^2}{|x|^2}-\sum_{1\leq j\neq k\leq3}\frac{x_jx_k}{|x|^2}\cdot|u'(r)|^2\frac{x_jx_k}{|x|^2} \notag \\
	&=\frac{1}{|x|^4}|u'(r)|^2\left\{\sum_{j=1}^3x_j^2(|x|^2-x_j^2)-\sum_{1\leq j\neq k\leq3}x_j^2x_k^2\right\}=0. \label{051}
\end{align}
Thus, we obtain
\begin{align*}
c\|\chi_Ru\|_{L^{p+1}}^{p+1}
	&\leq \frac{d}{dt}M(t)-\int_{\mathbb{R}^3}\frac{8}{R^2}\chi_R(\Delta\chi)\left(\frac{x}{R}\right)|u|^2dx+\frac{12(p-1)}{p+1}\int_{\frac{R}{2}\leq|x|}|u|^{p+1}dx\\
	&\hspace{4.5cm}+\int_{R\leq|x|\leq 2R}(\Delta\Delta w_R)|u|^2+|u|^{p+1}\Delta w_Rdx+o_R(1)\\
	&\lesssim \frac{d}{dt}M(t)+\frac{1}{R^2}M[u_0]+\int_{\frac{R}{2}\leq|x|}|u|^{p+1}dx+o_R(1).
\end{align*}
Integrating both sides of this inequality over $[0,T]$,
\begin{align*}
\int_0^Tc\|\chi_Ru\|_{L^{p+1}}^{p+1}dt\lesssim\sup_{t\in[0,T]}|M(t)|+\frac{T}{R^2}M[u_0]+\int_0^T\int_{\frac{R}{2}\leq|x|}|u|^{p+1}dx+o_R(1)T.
\end{align*}
By \eqref{074},
\begin{align*}
c\int_0^T\int_{|x|\leq\frac{R}{2}}|u|^{p+1}dxdt\lesssim R+\frac{T}{R^2}M[u_0]+\int_0^T\int_{\frac{R}{2}\leq|x|}|u|^{p+1}dxdt+o_R(1)T.
\end{align*}
Here, using Lemma \ref{Radial Sobolev embedding},
\begin{align*}
\int_{\frac{R}{2}\leq|x|}|u|^{p+1}dx\leq\||\cdot|u\|_{L^\infty}^{p-1}\int_{\frac{R}{2}\leq|x|}\frac{1}{|x|^{p-1}}|u|^2dx\lesssim\frac{1}{R^{p-1}}\|u\|_{H^1}^{p-1}M[u_0]\lesssim_u\frac{1}{R^{p-1}}M[u_0].
\end{align*}
Therefore,
\begin{align*}
\frac{1}{T}\int_0^T\int_{|x|\leq\frac{R}{2}}|u|^{p+1}dxdt\lesssim_{u,\delta}\frac{R}{T}+\frac{1}{R^2}+\frac{1}{R^{p-1}}+o_R(1),
\end{align*}
which completes the proof of this proposition.
\end{proof}

\begin{proposition}[Potential energy evacuation]\label{Energy evacuation}
Let $u$ be a solution to (NLS$_V$) satisfying the condition in Theorem \ref{Scattering versus blowing-up or growing-up} (i). Then, there exist sequences $\{t_n\}$ with $t_n\rightarrow\infty$ and $\{R_n\}$ with $R_n\rightarrow\infty$ such that
\begin{align*}
\liminf_{n\rightarrow\infty}\int_{|x|\leq R_n}|u(t_n,x)|^{p+1}dx=0.
\end{align*}
\end{proposition}

\begin{proof}
Applying Proposition \ref{Virial/Morawetz estimate} with $T=R^3$ implies
\begin{align}
\frac{1}{R^3}\int_0^{R^3}\int_{|x|\leq\frac{R}{2}}|u(t,x)|^{p+1}dxdt\lesssim\frac{1}{R^2}+\frac{1}{R^{p-1}}+o_R(1). \label{058}
\end{align}
By contradiction, we will prove
\begin{align}
\liminf_{t\rightarrow\infty}\int_{|x|\leq\frac{1}{2}t^\frac{1}{3}}|u(t,x)|^{p+1}dx=0. \label{059}
\end{align}
We assume that
\begin{align*}
\liminf_{t\rightarrow\infty}\int_{|x|\leq\frac{1}{2}t^\frac{1}{3}}|u(t,x)|^{p+1}dx=:\alpha>0.
\end{align*}
%that is,
%\begin{align*}
%\lim_{t\rightarrow\infty}\inf_{s>t}\int_{|x|\leq\frac{1}{2}s^\frac{1}{3}}|u(s,x)|^{p+1}dx=\alpha>0.
%\end{align*}
Then, there exists $t_0>0$ such that
\begin{align*}
\inf_{s>t}\int_{|x|\leq\frac{1}{2}s^\frac{1}{3}}|u(s,x)|^{p+1}dx>\frac{\alpha}{2}>0
\end{align*}
for any $t>t_0$. Therefore, we have
\begin{align*}
	&\frac{1}{R^3}\int_0^{R^3}\int_{|x|\leq\frac{1}{2}t^\frac{1}{3}}|u(t,x)|^{p+1}dxdt\\
	&\hspace{3.0cm}=\frac{1}{R^3}\int_0^{t_0}\int_{|x|\leq\frac{1}{2}t^\frac{1}{3}}|u(t,x)|^{p+1}dxdt+\frac{1}{R^3}\int_{t_0}^{R^3}\int_{|x|\leq\frac{1}{2}t^\frac{1}{3}}|u(t,x)|^{p+1}dxdt\\
	&\hspace{3.0cm}>\frac{1}{R^3}\int_{t_0}^{R^3}\frac{\alpha}{2}dt=\frac{R^3-t_0}{R^3}\cdot\frac{\alpha}{2}\longrightarrow\frac{\alpha}{2}>0\ \ \ \text{as}\ \ \ R\rightarrow\infty.
\end{align*}
This is contradiction with
\begin{align*}
\frac{1}{R^3}\int_0^{R^3}\int_{|x|\leq\frac{1}{2}t^\frac{1}{3}}|u(t,x)|^{p+1}dxdt\lesssim\frac{1}{R^2}+\frac{1}{R^{p-1}}+o_R(1)\longrightarrow0\ \ \ \text{as}\ \ \ R\rightarrow\infty,
\end{align*}
where we have used \eqref{058}. Consequently, by \eqref{059}, we can take $\{t_n\}:t_n\rightarrow\infty$ and $\{R_n\}:R_n=\frac{1}{2}t_n^\frac{1}{3}\rightarrow\infty$ such that
\begin{align*}
\lim_{n\rightarrow\infty}\int_{|x|\leq R_n}|u(t_n,x)|^{p+1}dx=0,
\end{align*}
which completes the proof of this proposition.
\end{proof}

Finally, we prove Theorem \ref{Scattering versus blowing-up or growing-up} (i).

\begin{proof}[Proof of Theorem \ref{Scattering versus blowing-up or growing-up} (i)]
By Proposition \ref{Coercivity1} (i), $u$ is globally in time and uniformly bounded in $H^1$. Fix $\varepsilon$ and $R$ as in Theorem \ref{Scattering criterion}. Now take sequences $\{t_n\}$ and $\{R_n\}$ satisfying $t_n\rightarrow\infty$ and $R_n\rightarrow\infty$ as in Proposition \ref{Energy evacuation}. Then, by choosing $n$ large enough such that $R_n\geq R$, H\"older's inequality and Proposition \ref{Energy evacuation} give
\begin{align*}
\int_{|x|\leq R}|u(t_n,x)|^2dx
	&\leq\left(\int_{|x|\leq R}dx\right)^\frac{p-1}{p+1}\left(\int_{|x|\leq R_n}|u(t_n,x)|^{p+1}dx\right)^\frac{2}{p+1}\\
	&\lesssim R^\frac{3(p-1)}{p+1}\left(\int_{|x|\leq R_n}|u(t_n,x)|^{p+1}dx\right)^\frac{2}{p+1}\longrightarrow0\ \ \text{as}\ \ n\rightarrow\infty.
\end{align*}
\end{proof}

Applying Theorem \ref{Scattering criterion}, $u$ scatters in positive time.

\section{Proof of blowing-up or growing-up part in Theorem \ref{Scattering versus blowing-up or growing-up}}

We prove the blows-up or grows-up part of Theorem \ref{Scattering versus blowing-up or growing-up} in this section.

\begin{lemma}[Coercivity \Rnum{2}]\label{Coercivity2}
Let $\frac{7}{3}<p<5$, $V\geq0$, and ``$V\in\mathcal{K}_0\cap L^\frac{3}{2}$ or $V\in L^\sigma$ for some $\sigma>\frac{3}{2}$''. Let $u$ be a solution to (NLS$_V$) with a data $u_0\in H^1$. Assume that $u_0$ satisfies
\begin{align*}
M[u_0]^{1-s_c}E_V[u_0]^{s_c}<(1-\delta)M[Q]^{1-s_c}E_0[Q]^{s_c}
\end{align*}
for some $\delta>0$ and
\begin{align*}
\|u_0\|_{L^2}^{1-s_c}\|\mathcal{H}^\frac{1}{2}u_0\|_{L^2}^{s_c}>\|Q\|_{L^2}^{1-s_c}\|\nabla Q\|_{L^2}^{s_c}.
\end{align*}
Then, there exists $\delta'=\delta'(\delta)>0$ such that
\begin{align*}
\|u(t)\|_{L^2}^{1-s_c}\|\mathcal{H}^\frac{1}{2}u(t)\|_{L^2}^{s_c}>(1+\delta')\|Q\|_{L^2}^{1-s_c}\|\nabla Q\|_{L^2}^{s_c}
\end{align*}
for any $t\in(T_\text{min},T_\text{max})$.
\end{lemma}

\begin{proof}
By $V\geq0$ and Gagliardo-Nirenberg's inequality, we have
\begin{align*}
(1-\delta)^\frac{1}{s_c}M
	&[Q]^\frac{1-s_c}{s_c}E_0[Q]>M[u_0]^\frac{1-s_c}{s_c}E_V[u_0]\\
%	&=\|u(t)\|_{L^2}^\frac{2(1-s_c)}{s_c}\left(\frac{1}{2}\|\nabla u(t)\|_{L^2}^2+\frac{1}{2}\int_{\mathbb{R}^3}V(x)|u(t,x)|^2dx-\frac{1}{p+1}\|u(t)\|_{L^p+1}^{p+1}\right)\\
	&\geq\|u(t)\|_{L^2}^\frac{2(1-s_c)}{s_c}\left(\frac{1}{2}\|\mathcal{H}^\frac{1}{2}u(t)\|_{L^2}^2-\frac{1}{p+1}C_\text{GN}\|u(t)\|_{L^2}^\frac{5-p}{2}\|\nabla u(t)\|_{L^2}^\frac{3(p-1)}{2}\right)\\
	&\geq\frac{1}{2}\|u(t)\|_{L^2}^\frac{2(5-p)}{3p-7}\|\mathcal{H}^\frac{1}{2}u(t)\|_{L^2}^2-\frac{1}{p+1}C_\text{GN}\|u(t)\|_{L^2}^\frac{3(5-p)(p-1)}{2(3p-7)}\|\mathcal{H}^\frac{1}{2}u(t)\|_{L^2}^\frac{3(p-1)}{2}\\
	&=\frac{1}{2}\|u(t)\|_{L^2}^\frac{2(5-p)}{3p-7}\|\mathcal{H}^\frac{1}{2}u(t)\|_{L^2}^2-\frac{2}{3(p-1)}\cdot\frac{\|u(t)\|_{L^2}^\frac{3(5-p)(p-1)}{2(3p-7)}\|\mathcal{H}^\frac{1}{2}u(t)\|_{L^2}^\frac{3(p-1)}{2}}{\|Q\|_{L^2}^\frac{5-p}{2}\|\nabla Q\|_{L^2}^\frac{3p-7}{2}}
\end{align*}
and hence,
\begin{align*}
(1-\delta)^\frac{1}{s_c}\geq\frac{3(p-1)}{3p-7}\cdot\frac{\|u(t)\|_{L^2}^\frac{2(5-p)}{3p-7}\|\mathcal{H}^\frac{1}{2}u(t)\|_{L^2}^2}{\|Q\|_{L^2}^\frac{2(5-p)}{3p-7}\|\nabla Q\|_{L^2}^2}-\frac{4}{3p-7}\cdot\frac{\|u(t)\|_{L^2}^\frac{3(5-p)(p-1)}{2(3p-7)}\|\mathcal{H}^\frac{1}{2}u(t)\|_{L^2}^\frac{3(p-1)}{2}}{\|Q\|_{L^2}^\frac{3(5-p)(p-1)}{2(3p-7)}\|\nabla Q\|_{L^2}^\frac{3(p-1)}{2}}.
\end{align*}
Here, we consider a function $g(y)=\frac{3(p-1)}{3p-7}y^2-\frac{4}{3p-7}y^\frac{3(p-1)}{2}$. Then, $g'(y)=\frac{6(p-1)}{3p-7}y-\frac{6(p-1)}{3p-7}y^\frac{3p-5}{2}$. Solving $g'(y)=0$, we obtain $y=0,1$. We set $y_0=0$ and $y_1=1$. Then, $g$ has a local minimum at $y_0$ and a local maximum at $y_1$. Also, $g(y_1)=1$. Combining these facts and the assumption of Proposition \ref{Coercivity2}, there exists $\delta'=\delta'(\delta)>0$ such that
\begin{align*}
\|u(t)\|_{L^2}^\frac{5-p}{3p-7}\|\mathcal{H}^\frac{1}{2}u(t)\|_{L^2}>(1+\delta')^\frac{2(p-1)}{3p-7}\|Q\|_{L^2}^\frac{5-p}{3p-7}\|\nabla Q\|_{L^2},
\end{align*}
%that is,
%\begin{align*}
%\|u(t)\|_{L^2}^{1-s_c}\|\mathcal{H}^\frac{1}{2}u(t)\|_{L^2}^{s_c}>(1+\delta')\|Q\|_{L^2}^{1-s_c}\|\nabla Q\|_{L^2}^{s_c},
%\end{align*}
which completes the proof of Lemma \ref{Coercivity2} and implies the uniform estimate \eqref{056} in Theorem \ref{Scattering versus blowing-up or growing-up} \rm{(ii)}.
\end{proof}

\begin{lemma}\label{Lemma1 for blows-up or grows-up}
Let $1\leq p<5$. Let $V$ satisfy ``$V\in\mathcal{K}_0\cap L^\frac{3}{2}$ and $\|V_-\|_{\mathcal{K}}<4\pi$'' or $V\in L^\sigma$ for some $\sigma>\frac{3}{2}$. We assume that $u\in C([0,\infty);H^1)$ be a solution to (NLS$_V$) satisfying
\begin{align*}
C_0:=\sup_{t\in[0,\infty)}\|\nabla u\|_{L^2}<\infty.
\end{align*}
Then, there exists $C_1>0$ such that for any $\eta>0$, $R>0$, and $t\in\left[0,\frac{\eta R}{\|u\|_{L^2}C_0C_1}\right]$,
\begin{align*}
\int_{|x|>R}|u(t,x)|^2dx\leq o_R(1)+\eta.
\end{align*}
\end{lemma}

\begin{proof}
Let $\Phi_R$ be a radial function constructed by $\Phi_R(x)=\Phi\left(\frac{x}{R}\right)$ and
\begin{equation}
\notag \Phi(x)=
\begin{cases}
&\hspace{0cm}0\hspace{1.25cm}(0\leq|x|\leq\frac{1}{2}),\\
&\hspace{-0.4cm}\text{smooth}\hspace{0.6cm}(\frac{1}{2}\leq|x|\leq1),\\
&\hspace{0cm}1\hspace{1.25cm}(1\leq|x|).
\end{cases}
\end{equation}
We note that there exists $C_1>0$ such that $|\nabla\Phi|\leq C_1$. We define a function
\begin{align*}
I(t):=\int_{\mathbb{R}^3}\Phi_R(x)|u(t,x)|^2dx.
\end{align*}
Using Proposition \ref{Virial identity},
\begin{align*}
I(t)
	&=I(0)+\int_0^t\frac{d}{ds}I(s)ds\leq I(0)+\int_0^t|I'(s)|ds\\
	&\leq I(0)+t\|\nabla\Phi_R\|_{L^\infty}\sup_{t\in[0,\infty)}\|\nabla u(t)\|_{L^2}\|u\|_{L^2}\leq I(0)+\frac{\|u\|_{L^2}C_0C_1t}{R}
\end{align*}
for any $t\in[0,\infty)$. Since $u_0\in H^1$,
\begin{align*}
I(0)=\int_{\mathbb{R}^3}\Phi_R(x)|u_0(x)|^2dx\leq\int_{|x|>\frac{R}{2}}|u_0(x)|^2dx=o_R(1).
\end{align*}
Moreover, by
\begin{align*}
\int_{|x|>R}|u(t,x)|^2dx\leq I(t),
\end{align*}
we have
\begin{align*}
\int_{|x|>R}|u(t,x)|^2dx\leq o_R(1)+\eta.
\end{align*}
\end{proof}

\begin{lemma}\label{Lemma2 for blows-up or grows-up}
Let $1<p<5$, ``$V\in\mathcal{K}_0\cap L^\frac{3}{2}$ or $V\in L^\sigma$ for some $\sigma>\frac{3}{2}$'', $V\geq0$, $x\cdot\nabla V+2V\geq0$, and $x\cdot\nabla V\in L^\frac{3}{2}$. Let $u\in C([0,\infty);H^1)$ be a solution to (NLS$_V$). We consider $\Psi_R$ be a radial function constructed by $\Psi_R(x)=R^2\Psi(\frac{x}{R})$ and
\begin{equation}
\notag \Psi=
\begin{cases}
&\hspace{-0.08cm}|x|^2\hspace{1.03cm}(0\leq|x|\leq1),\\
&\hspace{-0.4cm}\text{smooth}\hspace{0.7cm}(1\leq|x|\leq3)\\
&\hspace{0.1cm}0\hspace{1.23cm}(3\leq|x|)
\end{cases}
\end{equation}
with $\Psi''(r)\leq2$. We define a function
\begin{align*}
I(t):=\int_{\mathbb{R}^3}\Psi_R(x)|u(t,x)|^2dx.
\end{align*}
Then, for $q>p+1$, there exist constants $C=C(q,\|u_0\|_{L^2},C_0)>0$ and $\theta_q>0$ such that for any $R>0$ and $t\in[0,\infty)$, the estimate
\begin{align*}
	&I''(t)\leq8\left\{\|\nabla u(t)\|_{L^2}^2-\frac{3(p-1)}{2(p+1)}\|u(t)\|_{L^{p+1}}^{p+1}+\int_{\mathbb{R}^3}V(x)|u(t,x)|^2dx\right\}\\
	&\hspace{3cm}+C\,\|u(t)\|_{L^2(R\leq|x|)}^{(p+1)\theta_q}+\frac{C}{R^2}\|u(t)\|_{L^2(R\leq|x|)}^2+C\|x\cdot\nabla V\|_{L^\frac{3}{2}(R\leq|x|\leq3R)}
\end{align*}
holds, where $\theta_q:=\frac{2\{q-(p+1)\}}{(p+1)(q-2)}\in(0,\frac{2}{p+1}]$ and $C_0$ is given in Lemma \ref{Lemma1 for blows-up or grows-up}.
\end{lemma}

\begin{proof}
Using Proposition \ref{Virial identity}, we have
\begin{align*}
I''(t)=8\|\nabla u(t)\|_{L^2}^2-\frac{12(p-1)}{p+1}\|u(t)\|_{L^{p+1}}^{p+1}+8\int_{\mathbb{R}^3}V(x)|u(t,x)|^2dx+\mathcal{R}_1+\mathcal{R}_2+\mathcal{R}_3+\mathcal{R}_4,
\end{align*}
where $\mathcal{R}_k=\mathcal{R}_k(t)$ $(k=1,2,3,4)$ are defined by
\begin{align*}
\mathcal{R}_1
	&:=4\int_{\mathbb{R}^3}\left\{\frac{1}{r^2}\Psi''\left(\frac{r}{R}\right)-\frac{R}{r^3}\Psi'\left(\frac{r}{R}\right)\right\}|x\cdot\nabla u|^2dx+4\int_{\mathbb{R}^3}\left\{\frac{R}{r}\Psi'\left(\frac{r}{R}\right)-2\right\}|\nabla u(t,x)|^2dx,\\
\mathcal{R}_2
	&:=-\frac{2(p-1)}{p+1}\int_{\mathbb{R}^3}\left\{\Psi''\left(\frac{r}{R}\right)+\frac{2R}{r}\Psi'\left(\frac{r}{R}\right)-6\right\}|u(t,x)|^{p+1}dx,\\
\mathcal{R}_3
	&:=-\int_{\mathbb{R}^3}\left\{\frac{1}{R^2}\Psi^{(4)}\left(\frac{r}{R}\right)+\frac{4}{Rr}\Psi^{(3)}\left(\frac{r}{R}\right)\right\}|u(t,x)|^2dx,\\
\mathcal{R}_4
	&:=-8\int_{\mathbb{R}^3}V(x)|u(t,x)|^2dx-2\int_{\mathbb{R}^3}\frac{R}{r}\Psi'\left(\frac{r}{R}\right)(x\cdot\nabla V)|u(t,x)|^2dx.
\end{align*}
If $\frac{1}{r^2}\Psi''\left(\frac{r}{R}\right)-\frac{R}{r^3}\Psi'\left(\frac{r}{R}\right)\leq0$, then we have $\mathcal{R}_1\leq0$ by $\Psi'(\frac{r}{R})\leq\frac{2r}{R}$. If $\frac{1}{r^2}\Psi''\left(\frac{r}{R}\right)-\frac{R}{r^3}\Psi'\left(\frac{r}{R}\right)\geq0$, then
\begin{align*}
\mathcal{R}_1&\leq4\int_{\mathbb{R}^3}\left\{\Psi''\left(\frac{r}{R}\right)-\frac{R}{r}\Psi'\left(\frac{r}{R}\right)\right\}|\nabla u(t,x)|^2dx+4\int_{\mathbb{R}^3}\left\{\frac{R}{r}\Psi'\left(\frac{r}{R}\right)-2\right\}|\nabla u(t,x)|^2dx\\
&=4\int_{\mathbb{R}^3}\left\{\Psi''\left(\frac{r}{R}\right)-2\right\}|\nabla u(t,x)|^2dx\leq0.
\end{align*}
Next, we estimate $\mathcal{R}_2$. We note that since $q\geq2$, the Gagliardo-Nirenberg's inequality can be applied to get
\begin{align}
\|f\|_{L^q}\leq C\|f\|_{L^2}^{\frac{6-q}{2q}}\|\nabla f\|_{L^2}^{\frac{3(q-2)}{2q}}\label{077}
\end{align}
for any $f\in H^1(\mathbb{R}^3)$, where $C$ depends only on $q$. Thus due to $1<p$, for any $q\in[p+1,6]$, by the mass conservation law and the estimate \eqref{077}, we have
\begin{align*}
\sup_{t\in[0,\infty)}\|u(t)\|_{L^q}\leq C\sup_{t\in[0,\infty)}\|u(t)\|_{L^2}^\frac{6-q}{2q}\|\nabla u(t)\|_{L^2}^\frac{3(q-2)}{2q}\leq c\|u_0\|_{L^2}^\frac{6-q}{2q}C_0^\frac{3(q-2)}{2q}=:C_3.
\end{align*}
By this inequality and H\"older's inequality,
\begin{align*}
\mathcal{R}_2
%	&=-\frac{2(p-1)}{p+1}\int_{R\leq|x|}\left\{\Psi''\left(\frac{r}{R}\right)+\frac{2R}{r}\Psi'\left(\frac{r}{R}\right)-6\right\}|u(t,x)|^{p+1}dx\\
	&=-\frac{2(p-1)}{p+1}\int_{R\leq|x|\leq3R}\left\{\Psi''\left(\frac{r}{R}\right)+\frac{2R}{r}\Psi'\left(\frac{r}{R}\right)\right\}|u(t,x)|^{p+1}dx+\frac{12(p-1)}{p+1}\int_{R\leq|x|}|u(t,x)|^{p+1}dx\\
	&\leq C\|u(t)\|_{L^{p+1}(R\leq|x|)}^{p+1}\\
	&\leq C\|u(t)\|_{L^q(R\leq|x|)}^{(p+1)(1-\theta_q)}\|u(t)\|_{L^2(R\leq|x|)}^{(p+1)\theta_q}\\
	&\leq C\,C_3^{(p+1)(1-\theta_q)}\|u(t)\|_{L^2(R\leq|x|)}^{(p+1)\theta_q}.
\end{align*}
Moreover, we estimate $\mathcal{R}_3$.
\begin{align*}
\mathcal{R}_3=-\int_{R\leq|x|\leq3R}\left\{\frac{1}{R^2}\Psi^{(4)}\left(\frac{r}{R}\right)+\frac{4}{Rr}\Psi^{(3)}\left(\frac{r}{R}\right)\right\}|u(t,x)|^2dx\leq\frac{C}{R^2}\|u(t)\|_{L^2(R\leq|x|)}^2.
\end{align*}
Finally, we estimate $\mathcal{R}_4$. By $V\geq0$ and $x\cdot\nabla V+2V\geq0$, we have
\begin{align*}
\mathcal{R}_4
	&=-8\int_{\mathbb{R}^3}V(x)|u(t,x)|^2dx-2\int_{\mathbb{R}^3}\frac{R}{r}\Psi'\left(\frac{r}{R}\right)(x\cdot\nabla V)|u(t,x)|^2dx\\
	&\leq-4\int_{|x|\leq R}(x\cdot\nabla V+2V)|u(t,x)|^2dx-8\int_{R\leq|x|}V(x)|u(t,x)|^2dx\\
	&\hspace{6.0cm}-2\int_{R\leq|x|\leq3R}\frac{R}{r}\Psi'\left(\frac{r}{R}\right)(x\cdot\nabla V)|u(t,x)|^2dx\\
	&\leq-2\int_{R\leq|x|\leq3R}\frac{R}{r}\Psi'\left(\frac{r}{R}\right)(x\cdot\nabla V)|u(t,x)|^2dx\\
	&\leq C\|x\cdot\nabla V\|_{L^\frac{3}{2}(R\leq|x|\leq3R)}\|u(t)\|_{L^6(R\leq|x|\leq3R)}\\
	&\leq C\|x\cdot\nabla V\|_{L^\frac{3}{2}(R\leq|x|\leq3R)}\|\nabla u(t)\|_{L^2}\\
	&\leq C\|x\cdot\nabla V\|_{L^\frac{3}{2}(R\leq|x|\leq3R)},
\end{align*}
which completes the proof of the lemma.
\end{proof}

\begin{lemma}\label{Lemma3 for blows-up or grows-up}
Let $\frac{7}{3}<p<5$. Let $V$ satisfy ``$V\in\mathcal{K}_0\cap L^\frac{3}{2}$ and $\|V_-\|_{\mathcal{K}}<4\pi$'' or ``$V\in L^\sigma$ for some $\sigma>\frac{3}{2}$''. We assume that $x\cdot\nabla V+2V\geq0$, $u_0$ satisfies \eqref{008} and \eqref{056}. Then, there exists $\delta>0$ such that
\begin{align*}
\|\nabla u(t)\|_{L^2}^2-\frac{3(p-1)}{2(p+1)}\|u(t)\|_{L^{p+1}}^{p+1}
	&-\frac{1}{2}\int_{\mathbb{R}^3}x\cdot\nabla V(x)|u(t,x)|^2dx\\
	&\leq\|\nabla u(t)\|_{L^2}^2-\frac{3(p-1)}{2(p+1)}\|u(t)\|_{L^{p+1}}^{p+1}+\int_{\mathbb{R}^3}V(x)|u(t,x)|^2dx<-\delta.
\end{align*}
for any $t\in(T_\text{min},T_\text{max})$.
\end{lemma}

\begin{proof}
The left inequality holds since
\begin{align*}
-\frac{1}{2}\int_{\mathbb{R}^3}x\cdot\nabla V(x)|u(t,x)|^2dx
	&=-\frac{1}{2}\int_{\mathbb{R}^3}(x\cdot\nabla V+2V)|u(t,x)|^2dx+\int_{\mathbb{R}^3}V(x)|u(t,x)|^2dx\\
	&\leq\int_{\mathbb{R}^3}V(x)|u(t,x)|^2dx.%\ \ \ (\because\ x\cdot\nabla V+2V\geq0)
\end{align*}
For the second inequality,
\begin{align}
	\|\nabla u(t)\|_{L^2}^2-\frac{3(p-1)}{2(p+1)}\|u(t)\|_{L^{p+1}}^{p+1}+\int_{\mathbb{R}^3}V(x)|u(t,x)|^2dx
%	&\hspace{3.0cm}=\frac{3(p-1)}{2}E[u]-\frac{3p-7}{4}\|\nabla u\|_{L^2}^2-\frac{3p-7}{4}\int_{\mathbb{R}^3}V(x)|u(t,x)|^2dx \notag \\
	=\frac{3(p-1)}{2}E_V[u_0]-\frac{3p-7}{4}\|\mathcal{H}^\frac{1}{2}u(t)\|_{L^2}^2.\label{065}
\end{align}
By the assumption \eqref{008},
\begin{align*}
\varepsilon_1:=\frac{1}{2}\left\{\left(\frac{M[Q]}{M[u_0]}\right)^{\frac{1-s_c}{s_c}}E_0[Q]-E_V[u_0]\right\}>0
\end{align*}
and
\begin{align*}
E_V[u_0]
%=\frac{1}{2}E_V[u_0]+\frac{1}{2}E_V[u_0]
<\frac{1}{2}E_V[u_0]+\frac{1}{2}\left(\frac{M[Q]}{M[u_0]}\right)^\frac{1-s_c}{s_c}E_0[Q]=\left(\frac{M[Q]}{M[u_0]}\right)^\frac{1-s_c}{s_c}E_0[Q]-\varepsilon_1.
\end{align*}
Moreover by the estimate \eqref{056}, we have
\begin{align}
\|\mathcal{H}^\frac{1}{2}u(t)\|_{L^2}^2>\left(\frac{M[Q]}{M[u_0]}\right)^\frac{1-s_c}{s_c}\|\nabla Q\|_{L^2}^2\label{066}
\end{align}
for any $t\in(T_\text{min},T_\text{max})$. Therefore, 
\eqref{021}, \eqref{065}, and \eqref{066} give
\begin{align*}
	&\|\nabla u(t)\|_{L^2}^2-\frac{3(p-1)}{2(p+1)}\|u(t)\|_{L^{p+1}}^{p+1}+\int_{\mathbb{R}^3}V(x)|u(t,x)|^2dx\\
	&\hspace{3.0cm}<\frac{3(p-1)}{2}\left\{\left(\frac{M[Q]}{M[u_0]}\right)^\frac{1-s_c}{s_c}E_0[Q]-\varepsilon_1\right\}-\frac{3p-7}{4}\left(\frac{M[Q]}{M[u_0]}\right)^\frac{1-s_c}{s_c}\|\nabla Q\|_{L^2}^2\\
	&\hspace{3.0cm}=-\frac{3(p-1)}{2}\varepsilon_1=:-\delta.
\end{align*}
\end{proof}

\begin{proof}[Proof of blows-up or grows-up result in Theorem \ref{Scattering versus blowing-up or growing-up}]
We assume that
\begin{align*}
T_\text{max}=\infty\ \ \text{ and }\ \ \sup_{t\in[0,\infty)}\|\nabla u(t)\|_{L^2}<\infty
\end{align*}
for contradiction. By Lemma \ref{Lemma3 for blows-up or grows-up}, there exists $\delta>0$ such that
\begin{align*}
\|\nabla u(t)\|_{L^2}^2-\frac{3(p-1)}{2(p+1)}\|u(t)\|_{L^{p+1}}^{p+1}+\int_{\mathbb{R}^3}V(x)|u(t,x)|^2dx<-\delta
\end{align*}
for any $t\in[0,\infty)$. We consider the function $I(t)$ as Lemma \ref{Lemma2 for blows-up or grows-up}. From Lemma \ref{Lemma2 for blows-up or grows-up} and Lemma \ref{Lemma1 for blows-up or grows-up}, we have
\begin{align}
I''(s)
	&\leq-8\delta+C\,\|u(s)\|_{L^2(R\leq|x|)}^{(p+1)\theta_q}+\frac{C}{R^2}\|u(s)\|_{L^2(R\leq|x|)}^2+C\|x\cdot V\|_{L^\frac{3}{2}(R\leq|x|\leq3R)} \notag \\
	&\leq-8\delta+o_R(1)+C\eta^\frac{(p+1)\theta_q}{2}+\frac{C}{R^2}M[u_0]+C\|x\cdot V\|_{L^\frac{3}{2}(R\leq|x|\leq3R)} \notag \\
	&=-8\delta+C\eta^\frac{(p+1)\theta_q}{2}+o_R(1)\label{067}
\end{align}
for any $\eta>0$, $R>0$, and $s\in\left[0,\frac{\eta R}{\|u_0\|_{L^2}C_0C_1}\right]$. We take $\eta=\eta_0>0$ sufficiently small such as
\begin{align*}
C\eta_0^\frac{(p+1)\theta_q}{2}\leq2\delta.
\end{align*}
Then, \eqref{067} implies
\begin{align}
I''(s)\leq-6\delta+o_R(1)\label{068}
\end{align}
for any $R>0$ and $s\in\left[0,\frac{\eta_0R}{\|u_0\|_{L^2}C_0C_1}\right]$. We set
\begin{align*}
T=T(R):=\alpha_0R:=\frac{\eta_0R}{\|u_0\|_{L^2}C_0C_1}.
\end{align*}
Integrating \eqref{068} over $s\in[0,t]$ and integrating over $t\in[0,T]$, we have
\begin{align}
I(T)\leq I(0)+I'(0)T+\frac{1}{2}\left(-6\delta+o_R(1)\right)T^2=I(0)+I'(0)\alpha_0R+\frac{1}{2}\left(-6\delta+o_R(1)\right)\alpha_0^2R^2.\label{070}
\end{align}
Here, we can prove
\begin{align}
I(0)=o_R(1)R^2\ \ \text{ and }\ \ I'(0)=o_R(1)R. \label{071}
\end{align}
Indeed,
\begin{align*}
I(0)
	&=\int_{|x|\leq3R}\Psi_R(x)|u_0(x)|^2dx=\int_{|x|\leq\sqrt{R}}|x|^2|u_0(x)|^2dx+\int_{\sqrt{R}\leq|x|\leq3R}R^2\Psi\left(\frac{r}{R}\right)|u_0(x)|^2dx\\
	&\leq RM[u_0]+cR^2\int_{\sqrt{R}\leq|x|}|u_0(x)|^2dx=o_R(1)R^2,
\end{align*}
and
\begin{align*}
I'(0)
	&=2\text{Im}\int_{\mathbb{R}^3}R\Psi'\left(\frac{r}{R}\right)\frac{x\cdot\nabla u_0}{r}\overline{u_0}dx\\
	&=4\text{Im}\int_{|x|\leq\sqrt{R}}x\cdot\nabla u_0\overline{u_0}dx+2\text{Im}\int_{\sqrt{R}\leq|x|\leq3R}R\Psi'\left(\frac{r}{R}\right)\frac{x\cdot\nabla u_0}{r}\overline{u_0}dx\\
	&\leq4\sqrt{R}\,\|\nabla u_0\|_{L^2}\|u_0\|_{L^2}+cR\|\nabla u_0\|_{L^2}\|u_0\|_{L^2(\sqrt{R}\leq|x|)}=o_R(1)R.
\end{align*}
Combining \eqref{070} and \eqref{071}, we get
\begin{align*}
I(T)\leq (o_R(1)-3\delta\alpha_0^2)R^2.
\end{align*}
We take $R>0$ sufficiently large such as $o_R(1)-3\delta\alpha_0^2<0$. However, this is contradiction to
\begin{align*}
I(T)=\int_{\mathbb{R}^3}\Psi_R(x)|u(T,x)|^2dx\geq0.
\end{align*}
\end{proof}

\section{Proof of blowing-up part in Theorem \ref{Scattering versus blowing-up or growing-up}}

Finally, we prove the blowing-up part of Theorem \ref{Scattering versus blowing-up or growing-up} in this section.

\begin{proof}
We assume that $T_\text{max}=\infty$.\\
Let $xu_0\in L^2$.\\
Then, From Proposition \ref{Finite variance} and Lemma \ref{Lemma3 for blows-up or grows-up}, it follows that
\begin{align*}
\frac{d^2}{dt^2}\int_{\mathbb{R}^3}|x|^2|u(t,x)|^2dx<-8\delta<0.
\end{align*}
This is contradiction to
\begin{align*}
\int_{\mathbb{R}^3}|x|^2|u(t,x)|^2dx\geq0.
\end{align*}
Therefore, the solution $u$ blows-up.\\
Let $V$ and $u_0$ be radially symmetric.\\
We define radial functions
\begin{equation}
\notag F(r)=
\begin{cases}
&\hspace{-0.13cm}\displaystyle{\frac{1}{2}r^2\hspace{0.82cm}(r\leq1)},\\
&\hspace{-0.4cm}\displaystyle{\text{smooth}\hspace{0.5cm}(1<r<2)},\\
&\hspace{-0.05cm}\displaystyle{\frac{3}{2}r\hspace{0.91cm}(2\leq r)}.
\end{cases}
\end{equation}
satisfying $1-F''\geq0$ and $F_R(r)=R^2F\left(\frac{r}{R}\right)$. Also, we note that
\begin{align}
\left|-\frac{2(p-1)}{p+1}\left(-3+F''\left(\frac{r}{R}\right)+\frac{2R}{r}F'\left(\frac{r}{R}\right)\right)\right|\lesssim4\left(1-F''\left(\frac{r}{R}\right)\right)\label{069}
\end{align}
for any $r\geq0$. Using \eqref{025} in Proposition \ref{Virial identity},
\begin{align*}
	&2\text{Im}\frac{d}{dt}\int_{\mathbb{R}^3}\nabla F_R\cdot\nabla u\overline{u}dx\\
	&=4\int_{\mathbb{R}^3}F''\left(\frac{r}{R}\right)|\nabla u|^2dx-\frac{2(p-1)}{p+1}\int_{\mathbb{R}^3}\left\{F''\left(\frac{r}{R}\right)+\frac{2R}{r}F'\left(\frac{r}{R}\right)\right\}|u|^{p+1}dx\\
	&\hspace{0.5cm}-\int_{\mathbb{R}^3}\left\{\frac{1}{R^2}F^{(4)}\left(\frac{r}{R}\right)+\frac{4}{Rr}F^{(3)}\left(\frac{r}{R}\right)\right\}|u|^2dx-2\int_{\mathbb{R}^3}\frac{R}{r}F'\left(\frac{r}{R}\right)(x\cdot\nabla V)|u|^2dx\\
	&=4\int_{\mathbb{R}^3}\left(|\nabla u|^2-\frac{1}{2}x\cdot\nabla V|u|^2-\frac{3(p-1)}{2(p+1)}|u|^{p+1}\right)dx-\int_{\mathbb{R}^3}\left\{4\left(1-F''\left(\frac{r}{R}\right)\right)|u'|^2\right.\\
	&\hspace{0.5cm}\left.+\left(\frac{1}{R^2}F^{(4)}\left(\frac{r}{R}\right)+\frac{4}{Rr}F^{(3)}\left(\frac{r}{R}\right)\right)|u|^2+\frac{2(p-1)}{p+1}\left(-3+F''\left(\frac{r}{R}\right)+\frac{2R}{r}F'\left(\frac{r}{R}\right)\right)|u|^{p+1}\right\}dx\\
	&\hspace{0.5cm}+2\int_{\mathbb{R}^3}\left\{1-\frac{R}{r}F'\left(\frac{r}{R}\right)\right\}(x\cdot\nabla V)|u|^2dx\\
	&=4\int_{\mathbb{R}^3}\left(|\nabla u|^2-\frac{1}{2}x\cdot\nabla V|u|^2-\frac{3(p-1)}{2(p+1)}|u|^{p+1}\right)dx+2\int_{R\leq|x|}\left\{1-\frac{R}{r}F'\left(\frac{r}{R}\right)\right\}(x\cdot\nabla V)|u|^2dx\\
	&\hspace{0.5cm}-\int_{R\leq|x|}\left\{4\left(1-F''\left(\frac{r}{R}\right)\right)|u'|^2+\frac{2(p-1)}{p+1}\left(-3+F''\left(\frac{r}{R}\right)+\frac{2R}{r}F'\left(\frac{r}{R}\right)\right)|u|^{p+1}\right\}dx\\
	&\hspace{0.5cm}-\int_{R\leq|x|\leq 2R}\left(\frac{1}{R^2}F^{(4)}\left(\frac{r}{R}\right)+\frac{4}{Rr}F^{(3)}\left(\frac{r}{R}\right)\right)|u|^2dx.
\end{align*}
From Lemma \ref{Lemma3 for blows-up or grows-up},
\begin{align}
	&2\text{Im}\frac{d}{dt}\int_{\mathbb{R}^3}\nabla F_R\cdot\nabla u\overline{u}dx \notag \\
	&\leq4\int_{\mathbb{R}^3}\left(|\nabla u|^2+V|u|^2-\frac{3(p-1)}{2(p+1)}|u|^{p+1}\right)dx+2\int_{R\leq|x|}\left\{1-\frac{R}{r}F'\left(\frac{r}{R}\right)\right\}(x\cdot\nabla V)|u|^2dx \notag \\
	&\hspace{0.5cm}-\int_{R\leq|x|}\left\{4\left(1-F''\left(\frac{r}{R}\right)\right)|u'|^2+\frac{2(p-1)}{p+1}\left(-3+F''\left(\frac{r}{R}\right)+\frac{2R}{r}F'\left(\frac{r}{R}\right)\right)|u|^{p+1}\right\}dx \notag \\
	&\hspace{0.5cm}-\int_{R\leq|x|\leq 2R}\left(\frac{1}{R^2}F^{(4)}\left(\frac{r}{R}\right)+\frac{4}{Rr}F^{(3)}\left(\frac{r}{R}\right)\right)|u|^2dx \label{076} \\
	&<-4\delta+\frac{C}{R^2}M[u_0]+2\int_{R\leq|x|}\left\{1-\frac{R}{r}F'\left(\frac{r}{R}\right)\right\}(x\cdot\nabla V)|u|^2dx \notag \\
	&\hspace{0.5cm}-\int_{R\leq|x|}\left\{4\left(1-F''\left(\frac{r}{R}\right)\right)|u'|^2+\frac{2(p-1)}{p+1}\left(-3+F''\left(\frac{r}{R}\right)+\frac{2R}{r}F'\left(\frac{r}{R}\right)\right)|u|^{p+1}\right\}dx. \notag
\end{align}
In the following, we estimate the term
\begin{align*}
-\int_{R\leq|x|}\frac{2(p-1)}{p+1}\left(-3+F''\left(\frac{r}{R}\right)+\frac{2R}{r}F'\left(\frac{r}{R}\right)\right)|u|^{p+1}dx.
\end{align*}
Applying \eqref{069}, Lemma \ref{Radial Sobolev inequality}, and Young's inequality, we get
\begin{align*}
	&-\int_{R\leq|x|}\frac{2(p-1)}{p+1}\left(-3+F''\left(\frac{r}{R}\right)+\frac{2R}{r}F'\left(\frac{r}{R}\right)\right)|u|^{p+1}dx\\
	&\hspace{3.0cm}\lesssim\int_{R\leq|x|}4\left(1-F''\left(\frac{r}{R}\right)\right)|u|^{p+1}dx\\
	&\hspace{3.0cm}\lesssim4\int_R^\infty\left(1-F''\left(\frac{r}{R}\right)\right)|u(r)|^{p+1}r^2dr\\
	&\hspace{3.0cm}=4\int_R^\infty\int_R^r\frac{d}{ds}\left(1-F''\left(\frac{s}{R}\right)\right)ds|u(r)|^{p+1}r^2dr\\
	&\hspace{3.0cm}=4\int_R^\infty\int_s^\infty|u(r)|^{p+1}r^2dr\frac{d}{ds}\left(1-F''\left(\frac{s}{R}\right)\right)ds\\
	&\hspace{3.0cm}\lesssim4\int_R^\infty\int_{s\leq|x|}|u(x)|^{p+1}dx\frac{d}{ds}\left(1-F''\left(\frac{s}{R}\right)\right)ds\\
	&\hspace{3.0cm}\lesssim4\int_R^\infty\frac{1}{s^2}\|u\|_{L^2(s\leq|x|)}^\frac{p+3}{2}\|\nabla u\|_{L^2(s\leq|x|)}^\frac{p-1}{2}\frac{d}{ds}\left(1-F''\left(\frac{s}{R}\right)\right)ds\\
	&\hspace{3.0cm}=4\int_R^\infty\left(\varepsilon^{-\frac{p-1}{4}}s^{-2}\|u\|_{L^2(s\leq|x|)}^\frac{p+3}{2}\right)\left(\varepsilon^\frac{p-1}{4}\|\nabla u\|_{L^2(s\leq|x|)}^\frac{p-1}{2}\right)\frac{d}{ds}\left(1-F''\left(\frac{s}{R}\right)\right)ds\\
	&\hspace{3.0cm}\lesssim4\int_R^\infty\left(\varepsilon^{-\frac{p-1}{5-p}}s^{-\frac{8}{5-p}}\|u\|_{L^2(s\leq|x|)}^\frac{2(p+3)}{5-p}+\varepsilon\|\nabla u\|_{L^2(s\leq|x|)}^2\right)\frac{d}{ds}\left(1-F''\left(\frac{s}{R}\right)\right)ds\\
	&\hspace{3.0cm}\lesssim\varepsilon^{-\frac{p-1}{5-p}}R^{-\frac{8}{5-p}}\|u\|_{L^2(R\leq|x|)}^\frac{2(p+3)}{5-p}+\varepsilon\int_R^\infty\int_s^\infty|u'(r)|^2r^2dr\frac{d}{ds}\left(1-F''\left(\frac{s}{R}\right)\right)ds\\
	&\hspace{3.0cm}=\varepsilon^{-\frac{p-1}{5-p}}R^{-\frac{8}{5-p}}\|u\|_{L^2(R\leq|x|)}^\frac{2(p+3)}{5-p}+\varepsilon\int_R^\infty\int_R^r\frac{d}{ds}\left(1-F''\left(\frac{s}{R}\right)\right)ds|u'(r)|^2r^2dr\\
	&\hspace{3.0cm}=\varepsilon^{-\frac{p-1}{5-p}}R^{-\frac{8}{5-p}}\|u\|_{L^2(R\leq|x|)}^\frac{2(p+3)}{5-p}+\varepsilon\int_R^\infty\left(1-F''\left(\frac{r}{R}\right)\right)|u'(r)|^2r^2dr.
\end{align*}
Similarly, we have
\begin{align*}
2\int_{R\leq|x|}\left\{1-\frac{R}{r}F'\left(\frac{r}{R}\right)\right\}(x\cdot\nabla V)|u|^2dx
	&\leq\int_{R\leq|x|}\left(3-F''\left(\frac{r}{R}\right)-\frac{2R}{r}F'\left(\frac{r}{R}\right)\right)(x\cdot\nabla V)|u|^2dx\\
	&\lesssim\int_{R\leq|x|}\left(1-F''\left(\frac{r}{R}\right)\right)(x\cdot\nabla V)|u|^2dx\\
	&\lesssim \|x\cdot\nabla V\|_{L^\frac{3}{2}(R\leq|x|)}\int_{R\leq|x|}\left(1-F''\left(\frac{r}{R}\right)\right)|u'|^2dx.
\end{align*}
Therefore, we get
\begin{align*}
	&2\text{Im}\frac{d}{dt}\int_{\mathbb{R}^3}\nabla F_R\cdot\nabla u\overline{u}dx\leq-4\delta+\frac{C}{R^2}M[u_0]+C\varepsilon^{-\frac{p-1}{5-p}}R^{-\frac{8}{5-p}}M[u_0]^\frac{(p+3)}{5-p}\\
	&\hspace{3.0cm}+\left(C\varepsilon+C\|x\cdot\nabla V\|_{L^\frac{3}{2}(R\leq|x|)}-4\right)\int_{R\leq|x|}\left(1-F''\left(\frac{r}{R}\right)\right)|u'|^2dx.
\end{align*}
Thus, if we take $\varepsilon>0$ sufficiently small such as $C\varepsilon<2$ and then, we take $R>0$ sufficiently large such as
\begin{align*}
C\|x\cdot\nabla V\|_{L^\frac{3}{2}(R\leq|x|)}<\frac{3p-7}{4}<2\ \ \text{ and }\ \ \frac{C}{R^2}M[u_0]+C\varepsilon^{-\frac{p-1}{5-p}}R^{-\frac{8}{5-p}}M[u_0]^\frac{(p+3)}{5-p}<2\delta,
\end{align*}
it follows that
\begin{align*}
\text{Im}\frac{d}{dt}\int_{\mathbb{R}^3}\nabla F_R\cdot\nabla u\overline{u}dx<-\delta<0.
\end{align*}
Integrating this inequality over $[0,t)$,
\begin{align*}
\text{Im}\int_{\mathbb{R}^3}\nabla F_R\cdot\nabla u\overline{u}dx\leq -\delta t+\text{Im}\int_{\mathbb{R}^3}\nabla F_R\cdot\nabla u_0\overline{u_0}dx
\end{align*}
On the other hand, we have
\begin{align*}
\left|\text{Im}\int_{\mathbb{R}^3}\nabla F_R\cdot\nabla u\overline{u}dx\right|
	&\leq\int_{\mathbb{R}^3}|\nabla F_R||\nabla u||u|dx=R\int_{\mathbb{R}^3}\left|F'\left(\frac{r}{R}\right)\right||\nabla u||u|dx\leq CRM[u_0]^\frac{1}{2}\|\nabla u(t)\|_{L^2}.
\end{align*}
If we take $T_0$ satisfying
\begin{align*}
-\delta T_0+\text{Im}\int_{\mathbb{R}^3}\nabla F_R\cdot\nabla u_0\overline{u_0}dx<-\frac{1}{2}\delta T_0,
\end{align*}
then, from these inequalities, it follows that
\begin{align*}
\frac{1}{2}\delta t\leq-\text{Im}\int_{\mathbb{R}^3}\nabla F_R\cdot\nabla u\overline{u}dx\leq\left|\text{Im}\int_{\mathbb{R}^3}\nabla F_R\cdot\nabla u\overline{u}dx\right|\leq CRM[u_0]^\frac{1}{2}\|\nabla u(t)\|_{L^2}.
\end{align*}
Therefore, we have
\begin{align*}
Ct^2\leq\|\nabla u(t)\|_{L^2}^2
\end{align*}
for any $t\geq T_0$. From \eqref{076}, Lemma \ref{Radial Sobolev inequality}, and Young's inequality,
\begin{align*}
	&2\text{Im}\frac{d}{dt}\int_{\mathbb{R}^3}\nabla F_R\cdot\nabla u\overline{u}dx\\
	&\hspace{1.5cm}\leq4\int_{\mathbb{R}^3}\left(|\nabla u|^2+V|u|^2-\frac{3(p-1)}{2(p+1)}|u|^{p+1}\right)dx\\
	&\hspace{3.5cm}+C\|x\cdot\nabla V\|_{L^\frac{3}{2}(R\leq|x|)}\|\nabla u\|_{L^2}^2+\frac{C}{R^2}M[u_0]+\frac{C}{R^{p-1}}M[u_0]^\frac{p+3}{4}\|\nabla u\|_{L^2}^\frac{p-1}{2}\\
	&\hspace{1.5cm}=6(p-1)E_V[u_0]-(3p-7)\|\nabla u\|_{L^2}^2-(3p-7)\int_{\mathbb{R}^3}V|u|^2dx\\
	&\hspace{3.5cm}+C\|x\cdot\nabla V\|_{L^\frac{3}{2}(R\leq|x|)}\|\nabla u\|_{L^2}^2+\frac{C}{R^2}M[u_0]+\frac{C}{R^{p-1}}M[u_0]^\frac{p+3}{4}\|\nabla u\|_{L^2}^\frac{p-1}{2}\\
	&\hspace{1.5cm}\leq 6(p-1)E_V[u_0]-\frac{3p-7}{2}\|\nabla u\|_{L^2}^2+\frac{C}{R^2}M[u_0]+\frac{C}{R^\frac{4(p-1)}{5-p}}M[u_0]^\frac{p+3}{5-p}.
\end{align*}
Since $\|\nabla u\|_{L^2}^2\geq Ct^2$ and $E_V$, $M$ are independent of $t$, there exists $T_1\geq T_0$ such that
\begin{align*}
2\text{Im}\frac{d}{dt}\int_{\mathbb{R}^3}\nabla F_R\cdot\nabla u\overline{u}dx\leq -\frac{3p-7}{4}\|\nabla u(t)\|_{L^2}^2.
\end{align*}
Integrating this inequality over $[T_1,t]$,
\begin{align*}
\text{Im}\int_{\mathbb{R}^3}\nabla F_R\cdot\nabla u\overline{u}dx-\text{Im}\int_{\mathbb{R}^3}\nabla F_R\cdot\nabla u(T_1)\overline{u(T_1)}dx\leq -\frac{3p-7}{8}\int_{T_1}^t\|\nabla u(s)\|_{L^2}^2ds.
\end{align*}
Here, considering
\begin{align*}
\text{Im}\int_{\mathbb{R}^3}\nabla F_R\cdot\nabla u(T_1)\overline{u(T_1)}dx<-\frac{1}{2}\delta T_0<0,
\end{align*}
we get
\begin{align*}
\frac{3p-7}{8}\int_{T_1}^t\|\nabla u(s)\|_{L^2}^2ds\leq-\text{Im}\int_{\mathbb{R}^3}\nabla F_R\cdot\nabla u\overline{u}dx\leq CRM[u_0]^\frac{1}{2}\|\nabla u(t)\|_{L^2}.
\end{align*}
We set
\begin{align*}
S(t):=\int_{T_1}^t\|\nabla u(s)\|_{L^2}^2ds\ \ \text{ and }\ \ A:=\frac{1}{M[u_0]}\left(\frac{3p-7}{8CR}\right)^2.
\end{align*}
Then,
\begin{align*}
A\leq \frac{S'(t)}{S(t)^2}.
\end{align*}
Integrating this inequality over $[T_1+1,t)$,
\begin{align*}
A(t-T_1-1)\leq\frac{1}{S(T_1+1)}-\frac{1}{S(t)}\leq\frac{1}{S(T_1+1)}<\infty.
\end{align*}
However, this inequality is contradiction if we take a limit $t\rightarrow\infty$.
\end{proof}

\section{Appendix}

In this section, we prove Corollary \ref{Mass-critical result} by using Lemma \ref{Lemma1 for blows-up or grows-up} and Lemma \ref{Lemma2 for blows-up or grows-up} with $p=\frac{7}{3}$. We also use the following lemma, which is a slight modification of Lemma \ref{Lemma3 for blows-up or grows-up}.

\begin{lemma}\label{Lemma for mass-critical result}
Let $p=\frac{7}{3}$. Let $V$ satisfy ``$V\in\mathcal{K}_0\cap L^\frac{3}{2}$ and $\|V_-\|_{\mathcal{K}}<4\pi$'' or $V\in L^\sigma$ for some $\sigma>\frac{3}{2}$. We assume that $x\cdot\nabla V+2V\geq0$, $u_0$ satisfies $E_V[u_0]<0$. Then, we have
\begin{align*}
\|\nabla u(t)\|_{L^2}^2-\frac{3}{5}\|u(t)\|_{L^\frac{10}{3}}^\frac{10}{3}
	&-\frac{1}{2}\int_{\mathbb{R}^3}x\cdot\nabla V(x)|u(t,x)|^2dx\\
	&\leq\|\nabla u(t)\|_{L^2}^2-\frac{3}{5}\|u(t)\|_{L^\frac{10}{3}}^\frac{10}{3}+\int_{\mathbb{R}^3}V(x)|u(t,x)|^2dx=2E_V[u_0].
\end{align*}
for any $t\in(T_\text{min},T_\text{max})$, where $u$ is the solution to (NLS$_V$) on $(T_\text{min},T_\text{max})$.
\end{lemma}

\begin{proof}
The first inequality is proved by the same argument as the proof of Lemma \ref{Lemma3 for blows-up or grows-up}. The second identity is proved by the definition of the energy $E_V$.
\end{proof}

\begin{proof}[Proof of Corollary \ref{Mass-critical result}]
Corollary \ref{Mass-critical result} is deduced by the same argument as the proof of Theorem \ref{Scattering versus blowing-up or growing-up} (2). In the argument, Lemma \ref{Lemma1 for blows-up or grows-up}, Lemma \ref{Lemma2 for blows-up or grows-up}, and Lemma \ref{Lemma for mass-critical result} are used.
\end{proof}

\end{document}